\documentclass[11pt,leqno,british]{amsart}
\usepackage{ifpdf}

\usepackage{eucal}
\usepackage{url}
\usepackage[pagebackref]{hyperref}
\usepackage{amsfonts}
\usepackage{amsthm}
\usepackage{amsmath}
\usepackage{amscd}
\usepackage{amssymb}
\usepackage{graphicx}
\usepackage{rotating}
\usepackage{wasysym}
\usepackage[small,nohug,heads=LaTeX]{diagrams}
\diagramstyle[labelstyle=\scriptstyle]

\def\CC {{\mathbb C}}     
\def\HH {{\mathbb H}}     
\def\NN {{\mathbb N}}     
\def\PP {{\mathbb P}}     
\def\QQ {{\mathbb Q}}     
\def\RR {{\mathbb R}}     
\def\ZZ {{\mathbb Z}}     

\def\ring#1{\ifmmode \mathaccent'027 #1\else \rm\accent'027 #1\fi}

\newcommand{\re}{{\mathrm e}}

\newcommand{\ri}{{\mathrm i}}

\newcommand{\lift}[2]{%
\setlength{\unitlength}{1pt}
\begin{picture}(0,0)(0,0)
\put(0,{#1}){\makebox(0,0)[b]{${#2}$}}
\end{picture}
}
\newcommand{\lowerarrow}[1]{%
\setlength{\unitlength}{0.03\DiagramCellWidth}
\begin{picture}(0,0)(0,0)
\qbezier(-28,-4)(0,-18)(28,-4)
\put(0,-14){\makebox(0,0)[t]{$\scriptstyle {#1}$}}
\put(28.6,-3.7){\vector(2,1){0}}
\end{picture}
}
\newcommand{\upperarrow}[1]{%
\setlength{\unitlength}{0.03\DiagramCellWidth}
\begin{picture}(0,0)(0,0)
\qbezier(-28,11)(0,25)(28,11)
\put(0,21){\makebox(0,0)[b]{$\scriptstyle {#1}$}}
\put(28.6,10.7){\vector(2,-1){0}}
\end{picture}
}

\def\ul  {\underline}

\def\wt  {\widetilde}

\def\mc {\mathcal}

\def\Hom {\mathrm{Hom}}

\def\st {\mathrm{Stab}}

\def \bd {\begin{diagram}}
\def \ed {\end{diagram}}
\def\be  {\begin{eqnarray}}
\def\ee  {\end{eqnarray}}
\def\ben {\begin{eqnarray*}}
\def\een {\end{eqnarray*}}

\def\bpr {\begin{proof}[Proof]}
\def\epr {\end{proof}}
\def\bsp {\begin{split}}
\def\esp {\end{split}}
\def\bcd {\begin{CD}}
\def\ecd {\end{CD}}

\mathchardef\mhyphen="2D

\newcommand{\abs}[1]{\left\vert#1\right\vert}
\newcommand{\scal}[1]{\left\langle#1\right\rangle}

\newtheorem{theorem}{Theorem}[section]
\newtheorem{lemma}[theorem]{Lemma}
\newtheorem{prop}[theorem]{Proposition}
\newtheorem{coro}[theorem]{Corollary}
\newtheorem{remark}[theorem]{Remark}
\newtheorem{df}[theorem]{Definition}

\newtheorem{q}{Question}[section]
\newtheorem*{theorem2}{Theorem}

\theoremstyle{plain} 
\theoremstyle{plain}

\begin{document}

\title[Dynamical systems and categories]%
{Dynamical systems and categories}

\author{ G. Dimitrov, F. Haiden,  L. Katzarkov, M. Kontsevich }

\address[George Dimitrov]{Universit$\rm \ddot{a}$t Wien\\
Garnisongasse 3, 1090, Wien\\
\ $\rm \ddot{O}$sterreich}
\email{gkid@abv.bg}

\address[Fabian Haiden]{Universit$\rm \ddot{a}$t Wien\\
Garnisongasse 3, 1090, Wien\\
\ $\rm \ddot{O}$sterreich}
\email{fabian.haiden@univie.ac.at}

\address[Ludmil Katzarkov]
{Universit$\rm \ddot{a}$t Wien\\
Garnisongasse 3, 1090, Wien\\
\ $\rm \ddot{O}$sterreich}
\email{lkatzark@math.uci.edu}

\address[Maxim Kontsevich]
{IHES\\
G35 route de Chartres, F-91440\\
France}
\email{maxim@ihes.fr}

\begin{abstract}
We study questions motivated by results in the classical theory of dynamical
systems in the context of triangulated and $A_\infty$-categories.
First, entropy is defined for exact endofunctors and computed in a variety of
examples. In particular, the classical entropy of a pseudo-Anosov map is
recovered from the induced functor on the Fukaya category.
Second, the density of the set of phases of a Bridgeland stability condition is
studied and  a complete answer is given  in the case of bounded derived categories
of quivers. Certain  exceptional pairs in triangulated categories, which we call Kronecker pairs, are used to construct stability conditions with density of phases.
Some open questions and further directions are outlined as well.
\end{abstract}

\maketitle

\setcounter{tocdepth}{2}
\tableofcontents

\section{Introduction}

Recent work of Cantat-Lamy \cite{CL} on the Cremona group and Blanc-Cantat
\cite{BlancCantat} on dynamical spectra suggests that there exists a deep
parallel between the study of groups of birational automorphisms on one hand,
and mapping class groups on the other.
Under this parallel, the dynamical degree of a birational map plays the role of
the entropy of pseudo-Anosov maps.
In the present paper we consider these developments from the perspective of
derived categories and their groups of autoequivalences, making direct
connections with the classical theory.

In another striking series of papers
Gaiotto-Moore-Neitzke \cite{GNM} and Bridgeland-Smith \cite{BS}
have established a connection between Teichm\"uller theory and theory of
stability conditions on triangulated categories.
An analogy between Teichm\"uller geodesic flow and the space of stability
conditions had been noticed previously in the paper 
\cite{KS} by Soibelman and the fourth author. 
One of the results is a correspondence between geodesics and stable
objects, with slopes of the former giving the phases of the latter.
Motivated by this, we study the set of phases of stable objects for
general stability conditions, leading to some surprising results.

Let us describe the contents of the paper in more detail.
First, we define and study entropy in the context of triangulated
and $A_\infty$-categories, more specifically \emph{dynamical entropy}, as a measure
of complexity of a dynamical system.
This notion comes in a variety of flavors: Let us mention Kolmogorov-Sinai
(measure-theoretic) entropy \cite{sinai1959}, topological entropy
\cite{adler1965}, and algebraic entropy \cite{bellon1999}.
In analogy with these notions, we define in Section 2 the entropy of an
exact endofunctor of a triangulated category with generator.
By taking into account the $\mathbb Z$-graded nature of the category one
gets not just a single real number, but a function on the real line.
In the case of saturated (smooth and proper) $A_{\infty}$-categories, the
following foundational results are proven (see
Theorems~\ref{entropyPoincareSeries} and \ref{hom_entropy_bound} for the
precise statements)
\begin{theorem2} In the saturated case, the entropy of an
endofunctor may be computed as a limit of Poincar\'e polynomials of Ext-groups.
\end{theorem2}
\begin{theorem2}
In the saturated case (under a certain generic technical condition), there is a
lower bound on the entropy given by the logarithm of the spectral radius of the
induced action on Hochschild homology.
\end{theorem2}

We compute the entropy of endofunctors in various examples.
In the most basic case of semisimple categories the entropy of any endofunctor
is described, in log-coordinates, as a branch of a real algebraic curve defined
over $\mathbb Z$.
Next, we show --- under a certain generic technical condition --- that the
entropy of a regular endomorphism of a smooth projective variety is given by the
logarithm of the spectral radius of the induced map on cohomology.
As another direction, we determine the entropy of the Serre functor in a number
of cases.
The results suggest some relation to the ``dimension'' of the category.
Finally, we show that the entropy of an autoequivalence induced by a
pseudo-Anosov map on the Fukaya category of a surface coincides with its
topological entropy, the logarithm of the stretch factor.
This concludes Section 2 of the paper.

Next, in Section 3,   the topological properties of the set of
phases of a stability condition on a triangulated category are
investigated. This notion of stability was introduced by T.
Bridgeland \cite{Bridg1} based on a proposal by M. Douglas
\cite{Douglas1, Douglas2}. It is an essential ingredient in the
theory of motivic Donaldson-Thomas invariants developed by Y.
Soibelman and the fourth author \cite{KS}. As we have mentioned,
guided by work of Gaiotto, Moore, and Neitzke \cite{GNM},
Bridgeland and Smith \cite{BS} identify spaces of meromorphic
quadratic differentials with spaces of stability conditions on
categories associated with quivers with potential. Moreover, under
this correspondence, stable objects are identified with geodesics
of finite length. Motivated by the density of the set of slopes of
closed geodesics on a Riemann surface    we investigate in Section
3 the question whether a given triangulated category admits a
stability condition such that the set of phases of stable objects
is dense somewhere in the circle.

In case of stability conditions non-dense behavior is possible
(Lemma \ref{Dynkin} and Corollary \ref{coro for density1}):
 \begin{theorem2}The
phases are never dense in an arc for  Dynkin and Euclidean quivers.\footnote{i.
e. acyclic quivers with underlying graph  a Dynkin or an extended Dynkin diagram} \end{theorem2}

Similarly as in the  case of geodesics density  property  is expected to hold in general.

 We introduce the notion of Kronecker pair (Definition \ref{defKroneckeppair})
 and show that (Theorem \ref{coro for kronecker pairs 0}) \begin{theorem2}If a  $\CC$-linear
 triangulated category  $\mc T$ contains a Kronecker pair, s. t. a certain family of stability conditions on
 it is extendable to the whole category, then the extended stability conditions
have phases dense in some arc. \end{theorem2}  Using this theorem  we obtain
(Proposition \ref{non euc and non dyn}):
\begin{theorem2}Any connected  quiver $Q$, which is neither Euclidean nor Dynkin has a
 family of stability conditions with phases which are dense in an arc. \end{theorem2}
We record these findings on density property in the following table:
\begin{gather}  \label{table} \begin{array}{| c | c | c | c | c | c | c |}
  \hline
         \mbox{Dynkin quivers}            &   P_\sigma  \  \mbox{is always finite}                 \\ \hline
   \mbox{  Euclidean quivers}  &  P_\sigma  \  \mbox{is either finite  or} \ \mbox{has exactly two limit points}         \\ \hline
    \mbox{All other quivers} &   \ P_\sigma \ \mbox{is dense in an arc }   \mbox{for  a family of stability conditions}\\ \hline
     \end{array}
  \end{gather}
where $P_\sigma $ denotes the set of stable phases (see Definition
\ref{P_sigma}).

By the  non-dense behavior of  stability conditions  we find that
\emph{on Dynkin and Euclidean quivers the dimensions of  Hom
spaces of exceptional pairs
 are strictly smaller than $3$} (Corollary \ref{upper bound for euc}).

  Further  examples of density of phases (blow ups of projective spaces) are given in subsection \ref{further examples}.

We summarize the guiding analogies between the classical theory of
dynamical systems and the theory of triangulated and $A_\infty$-categories in
the following table.

\vskip 1\baselineskip

\begin{tabular}{| c | c | c | c | c |}
  \hline
  Classical       & \begin{tabular}{l} Geodesics dense set \\  \includegraphics[bb=0 0 50 30, width=1.9cm, height=2cm, scale=0.5]{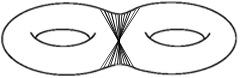} \end{tabular}&   Geodesics         &  \begin{tabular}{l} Thurston \\ compactification \\ \includegraphics[ width=3cm, height=2cm, scale=0.5]{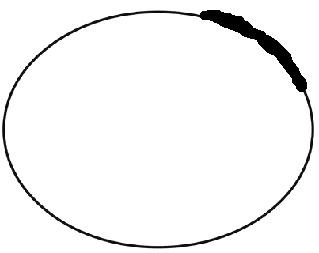}  \end{tabular}   & \begin{tabular}{l}  Entropy \\  \includegraphics[bb=0 0 80 80, width=3cm, height=2cm, scale=0.5]{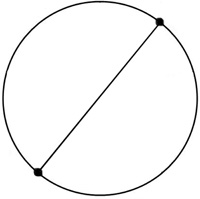} \end{tabular}            \\ \hline
  Categorical     &    \begin{tabular}{l} Density of phases \\
                      \includegraphics[bb=0 0 80 80, width=2cm, height=2cm, scale=0.5]{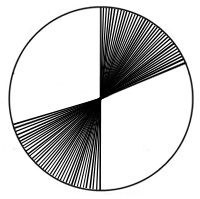} \end{tabular}         &    Stable objects   &  Stability conditions                  &  \begin{tabular}{c} Behaviour of\\ the mass of a generator
                                                                                                            \end{tabular} \\ \hline
\end{tabular}

\vskip 1\baselineskip

In this paper we only scratch the surface  of a rather promising,
in our opinion, area of future research. In the final section we
suggest some open problems and possible new directions related to
stability conditions, birational geometry, and Fukaya categories.
First, inspired by the classical theory of pseudo-Anosov maps, we
suggest a definition of a pseudo-Anosov functor in the context of
triangulated categories with stability conditions. This is
conjecturally a direct generalization of the usual notion which
has a geometric interpretation in higher dimensions. Furthermore,
we propose a connection with the works of Cantat-Lamy~\cite{CL}
and Blanc-Cantat~\cite{BlancCantat}. In \cite{CL}, based on
Gromov's ideas in geometric group theory, the authors prove the
nonsimplicity of Cremona group in dimension two. Later, in the
paper \cite{BlancCantat}, Blanc and Cantat consider the dynamical
spectra of the group of birational transformations as a tool for
studying groups of birational automorphisms of surfaces. We
propose that this dynamical spectrum is part of a categorical
dynamical spectrum. We also briefly explore Gromov's idea that
entropy measures growth of volume of submanifolds under the action
of a smooth map. From this perspective we conjecture that
categorical complexity is related to the mass of objects with
respect to a stability condition. Finally, we pose several
questions pertaining to the behavior of phases of stable objects.

{\bf Acknowledgements:} We are grateful to Denis Auroux for many useful discussions.
The authors were funded by NSF DMS 0854977 FRG, NSF DMS 0600800, NSF DMS 0652633
FRG, NSF DMS 0854977, NSF DMS 0901330, FWF P 24572 N25, by FWF P20778 and by an
ERC Grant.

\section{Entropy of endofunctors}

In this section we define and study the entropy of an exact endofunctor of a
triangulated category with generator.
Our definition, see Subsection \ref{subsecEntropy}, is based on a notion
of complexity of one object relative to another in such a category discussed in
Subsection \ref{subsecComplexity}.
In Subsection \ref{subsecSaturated} the case of saturated
$A_\infty$-categories is treated, for which entropy is more easily computed
from dimensions of $\mathrm{Ext}$-groups.
The rest of the section is devoted to various examples.

\subsection{Complexity in triangulated categories} \label{subsecComplexity}

Suppose $\mc T$ is a triangulated category with (split-) generator $G$.
By definition, this means that for every $E\in\mc T$ there is an $E'$ and a
tower of triangles
\begin{equation}
\begin{diagram}[size=1em]
0=F_0 & \rTo      &        &       &   F_1 & \rTo      &         &        & F_2
 &                  &F_{k-1} & \rTo      &        &       & F_k\cong E \oplus E'
 \\
      & \luDashto &        & \ldTo &       &\luDashto  &         &  \ldTo &
 & \quad\cdots\quad &        & \luDashto &        & \ldTo &       \\
      &           & G[n_1] &       &       &           &  G[n_2] &        &
 &                  &        &           & G[n_k] &
\end{diagram}
\end{equation}
with $k\ge 0$, $n_i\in\mathbb Z$.
One can ask, for each $E\in\mc T$, what the least number of shifted copies of
$G$ needed is, in order for such a relation to hold.
Taking into account the $n_i$'s as well, one arrives at the following
definition.
\begin{df} \label{complexity}
Let $E_1, E_2$ be objects in a triangulated category $\mc T$.
The \emph{complexity of $E_2$ relative to $E_1$} is the function
$\delta_t(E_1,E_2):\RR\to[0,\infty]$ given by
\begin{equation*}
\delta_t(E_1,E_2) = \inf \left \{ \sum_{i=1}^k \re^{n_i t} \left \vert  \begin{diagram}[size=0.8em] \label{HN filtration} 0 & \rTo      &     &       &   A_1 & \rTo      &     &        & A_2 &        &  A_{k-1} & \rTo      &     &       & E_2 \oplus E_2' \\
  & \luDashto &     & \ldTo &       &\luDashto  &     &  \ldTo &     &         \cdots    &              & \luDashto &     & \ldTo &       \\
  &           & E_1[n_1] &       &       &           &  E_1[n_2] &            &      &             &         &           & E_1[n_k] &
\end{diagram}  \right. \right\}.
\end{equation*}
\end{df}
Note that $\delta_t(E_1,E_2)=+\infty$ iff $E_2$ does not lie in the thick
triangulated subcategory generated by $E_1$.
Also, when $\mc T$ is $\mathbb Z/2$-graded in the sense that
$[2]\cong\mathrm{id}_{\mc T}$, only the value at zero, $\delta_0(E_1,E_2)$, will
be of any use.
We collect some basic inequalities in the following proposition.

\begin{prop} \label{complexityInequalities}
Let $\mc T$ be a triangulated category, $E_1,E_2,E_3\in\mc T$. Then

\emph{(a)} (Triangle inequality) $\delta_{t}(E_1, E_3)\leq \delta_{t}(E_1,
E_2)\delta_{t}(E_2, E_3)$,

\emph{(b)} (Subadditivity) $\delta_t(E_1,E_2\oplus E_3)\leq
\delta_t(E_1,E_2)+\delta_t(E_1,E_3)$,

\emph{(c)} (Retraction) $\delta_{t}(F(E_1), F(E_2)) \leq  \delta_{t}(E_1,
E_2)$ for any exact functor $F: \mc T \rightarrow \mc T'$.
\end{prop}

\bpr We will prove only (a).
Take $\epsilon>0$ and let
\ben
\begin{diagram}[size=1em] 0 & \rTo      &     &       &   A_1 & \rTo      &     &        & A_2 &        &  A_{k-1} & \rTo      &     &       & E_2 \oplus E_2' \\
  & \luDashto &     & \ldTo &       &\luDashto  &     &  \ldTo &     &         \cdots    &              & \luDashto &     & \ldTo &       \\
  &           & E_1[n_1] &       &       &           &  E_1[n_2] &            &      &             &         &           & E_1[n_k] &
\end{diagram}
\een
\ben
 \begin{diagram}[size=1em]
0 & \rTo      &     &       &   B_1 & \rTo      &     &        & B_2 &        &  B_{p-1} & \rTo      &     &       & E_3 \oplus E_3' \\
  & \luDashto &     & \ldTo &       &\luDashto  &     &  \ldTo &     &         \cdots    &              & \luDashto &     & \ldTo &       \\
  &           & E_2[m_1] &       &       &           &  E_2[m_2] &            &
  &             &         &           & E_2[m_p] &
\end{diagram}
\een
be such that
\be \sum_{i=1}^k \re^{n_i t} < \delta_t(E_1,E_2)+\epsilon,\qquad \ \sum_{j=1}^p
\re^{m_j t} < \delta_t(E_2,E_3)+\epsilon.\ \ee
Using the given sequences of triangles one can construct a sequence of the type
\ben \begin{diagram}[size=1em]
0 & \rTo      &     &       &   C_1 & \rTo      &     &        & C_2 &        &  C_{k p-1} & \rTo      &     &       & E_3 \oplus E_3'' \\
  & \luDashto &     & \ldTo &       &\luDashto  &     &  \ldTo &     &         \cdots    &              & \luDashto &     & \ldTo &       \\
  &           & E_1[n_1+m_1] &       &       &           &  E_1[n_2+m_1] &            &      &             &         &           & E_1[n_k+m_p] &
\end{diagram}, \een
where $E_3''= E_3'\oplus E_2'[m_1]\oplus E_2'[m_2] \dots \oplus   E_2'[m_p]$. Hence
\be \delta_t(E_1,E_3)\leq \sum_{i,j}\re^{(n_i+m_j) t} \leq
(\delta_t(E_1,E_2)+\epsilon )(\delta_t(E_2,E_3)+\epsilon ).\ee
Since this holds for any $\epsilon>0$, the inequality follows.
\epr

For complexes of vector spaces the complexity with respect to the ground field
is just the Poincar\'e (Laurent-) polynomial in $\re^{-t}$.

\begin{lemma} \label{complexity in Dbvec} Let $\mc T = D^b(k)$ be the bounded
derived category of finite-dimensional vector spaces over a field $k$.
Then
\ben \delta_t(k,E) = \sum_{n} \dim (H^n E)\re^{-nt} \een
for any $E\in\mc T$.
\end{lemma}

\bpr
Any object in $\mc T$ is isomorphic to its cohomology,
which is a direct sum with $\dim (H^n E)$ copies of $k[-n]$ for each $n$.
This shows the inequality ``$\leq$''.
For the reverse inequality, use the fact that for any exact triangle $E_1\to
E_2\to E_3\to E_1[1]$ we have $\dim H^nE_2\leq \dim H^nE_1+\dim H^n E_3$.
\epr

\subsection{Entropy of exact endofunctors} \label{subsecEntropy}

We now come to the main definition of this section.
It is based on the observation that if $F$ is an endofunctor of a triangulated
category $\mc T$ with generator $G$, then $\delta_t(G,F^nG)$ grows at most
exponentially, and the growth constant is independent of the choice of
generator.

\begin{df} \label{definition of entropy} Let  $F:\mc T \rightarrow \mc T$ be an
exact endofunctor of a triangulated category $\mc T$ with generator $G$.
The \emph{entropy of $F$} is the function $h_t(F):\RR\to[-\infty,+\infty)$ of
$t$ given by
\ben h_t(F)= \lim_{n\rightarrow \infty}\frac{1}{n}\log\delta_t(G,F^nG). \een
\end{df}

The next lemma shows that $h_t(F)$ is well-defined.

\begin{lemma}
With $\mc T, F$ as in the definition, the limit
$\lim_{n\rightarrow \infty}\frac{1}{n}\log(\delta_t(G,F^nG))$
exists in $[-\infty,+\infty)$ for every $t$ and is independent of the choice
of generator $G$.
\end{lemma}

\bpr
Using Proposition \ref{complexityInequalities},
\be \delta_t(G,F^{m+n}(G)) \leq \delta_t(G,F^{m}(G))
\delta_t(F^m(G),F^{m+n}(G)) \leq \delta_t(G,F^{m}(G))
\delta_t(G,F^{n}(G)) \ee
hence existence of the limit follows from Fekete's Lemma: for any
subadditive sequence $\{a_n\}_{n \geq 1}$ we have $\lim_{n\rightarrow
\infty}\frac{a_n}{n}=\inf\{\frac{a_n}{n}\vert n \geq 1\}$.
Furthermore, if $G,G'$ are generators then
\begin{align}
\delta_t(G',F^n(G')) &\leq
\delta_t(G',G) \delta_t(G,F^n(G))\delta_t(F^n(G),F^n(G')) \\
&\leq  \delta_t(G',G) \delta_t(G,G') \delta_t(G,F^n(G))
\end{align}
which implies the second claim.
\epr

We conclude this subsection with some remarks on the behaviour of
$h_t$ with respect to composition of functors. Clearly, $h_t$ is
constant on conjugacy classes, and $h_t(F^n)=nh_t(F)$ for $n\ge
1$. If $\mc T\neq 0$, then $h_0(\mathrm{id}_{\mc T})=0$, but this
may fail for other values of $t$. In general, nothing can be said
about $h_t(F\circ H)$ from $h_t(F)$ and $h_t(H)$ alone, but if $F$
and $H$ commute, then $h_t(F\circ H)\leq h_t(F)+h_t(H)$.

\subsection{The case of saturated \texorpdfstring{$A_{\infty}$}{\space}-categories}
\label{subsecSaturated}

For background on $A_{\infty}$-categories see for example \cite{lefevre} and
\cite{ks2009}.
We will work over a fixed field $k$.
An $A_\infty$-category $\mc C$ is \emph{triangulated} if every compact object in
$\mathrm{Mod}(\mc C)$ is quasi-representable.
The homotopy category, $H^0\mc C$, of a triangulated $A_\infty$-category is
triangulated in the sense of Verdier.
An $A_\infty$-category is \emph{saturated} if it is
triangulated and Morita-equivalent to a smooth and compact $A_\infty$-algebra
$A$.

In the case of saturated $A_\infty$-categories one can, for
the purpose of computing the entropy, replace $\delta_t(A,B)$ with the Poincar\'e
polynomial of $\mathrm{Ext}^*(A,B)$ in $\re^{-t}$.

\begin{theorem} \label{entropyPoincareSeries}
Let $\mc C$ be a saturated $A_\infty$-category and $F$ an endofunctor of $\mc
C$. Then
\ben
h_t(F)=\lim_{N\to\infty}\frac{1}{N}\log\sum_{n\in\ZZ}\dim\mathrm{Ext}^n(G,F^N
G)\re^{-nt}
\een
for any generator $G$ of $\mc C$.
\end{theorem}

\bpr
Note first that
\be \delta_t(k,R\Hom(G,F^NG))=\sum_{n}\dim\mathrm{Ext}^n(G,F^NG)\re^{-nt} \ee
by Lemma \ref{complexity in Dbvec}.
We will show that there exist $C_1,C_2:\mathbb R\to\mathbb R_{>0}$, depending on
$G$, such that
\be \label{distEstimate}
C_1(t)\delta_t(G,E) \leq \delta_t(k,R\Hom(G,E)) \leq C_2(t)\delta_t(G,E)
\ee
for every $E\in\mc C$.
The theorem follows from \eqref{distEstimate} by setting $E=F^NG$, taking
the logarithm, dividing by $N$, and passing to the limit $N\to\infty$.

For the second inequality, apply Proposition~\ref{complexityInequalities} to the
functor $R\Hom(G,\_):\mc C\to D^b(k)$, which is well defined by local properness
of $\mc C$.
We get
\begin{align}
\delta_t(k,R\Hom(G,E)) &\leq \delta_t(k,R\Hom(G,G))\delta_t(R\Hom(G,G),R\Hom(G,E))\\
  &\leq \delta_t(k,R\Hom(G,G))\delta_t(G,E)
\end{align}
where the first factor is a function $C_2(t):\mathbb R\to\mathbb R_{>0}$
independent of $E$.
Note that we did not use the assumption that $\mc C$ is smooth.

On the other hand
\begin{align}
\delta_t(G,E)
 &\leq\delta_t(G,G\otimes R\Hom(G,E)) \delta_t(G\otimes R\Hom(G,E),E)\\
 &\leq\delta_t(k,R\Hom(G,E)) \delta_t(G \otimes R\Hom(G,\_),\mathrm{id}_{\mc C})
\end{align}
by the retraction map property of the functors $G\otimes\_:D^b(k)\to \mc C$ and
$\Phi\mapsto \Phi(E):\mathrm{Fun}(\mc C,\mc C)\to\mc C$.
It remains to show that the second factor, $\delta_t(G \otimes
R\Hom(G,\_),\mathrm{id}_{\mc C})$ is in fact finite, i.e. that $\mathrm{id}_{\mc
C}$ lies in the subcategory of $\mathrm{Fun}(\mc C,\mc C)$ generated by $G\otimes R\Hom(G,\_)$.
But this is just equivalent to smoothness of $\mc C$.
Indeed, let $A=\mathrm{End}(G)$, then the functor $G\otimes R\Hom(G,\_)$
corresponds to the bimodule $A\otimes_k A^{\mathrm{op}}$ and $\mathrm{Id}_{\mc
C}$ corresponds to the diagonal bimodule $A$.
This completes the proof of \eqref{distEstimate} and the theorem.
\epr

Suppose that $\mc C$ is a saturated $A_\infty$-category over $\mathbb C$. Since
$\mc C$ is smooth and proper, its Hochschild homology, $HH_*(\mc C)$, is
finite-dimensional.
Any endofunctor $F$ of $\mc C$ induces a linear map  on $HH_*(\mc C)$
and we may consider the logarithm of its spectral radius as a kind of ``homological
entropy''.
We will show that, under a certain generic condition discussed in the lemma
below, this number is bounded above by $h_0(F)$.

\begin{lemma} \label{str_radius}
Let $V$ be a finite-dimensional $\mathbb{Z}/(2)$-graded $\mathbb C$-vector
space, $A=A_+\oplus A_-$ a homogeneous endomorphism of $V$.
Suppose that $A_+$, $A_-$ do not have the same eigenvalues, with multiplicity,
on $\{|\lambda|=\rho(A)\}$. Then
\ben \limsup_{n\to\infty}\left|\mathrm{sTr}A^n\right|^{1/n}=\rho(A) \een
\end{lemma}

\bpr
We may assume that $\rho(A)\neq 0$.
Let $\alpha_1,\ldots,\alpha_m$ be the eigenvalues of $A_+$ and
$\beta_1,\ldots,\beta_n$ those of $A_-$.
Consider the identity
\be \label{str_series} \sum_{n=0}^{\infty}\mathrm{sTr}A^nz^n=
\sum_{k=1}^m\frac{1}{1-\alpha_kz}-\sum_{l=1}^n\frac{1}{1-\beta_lz}\ee
which holds for $|z|$ sufficiently small.
If $R$ denotes the radius of convergence of the series on the left
hand side of \eqref{str_series}, then
\be R^{-1}=\limsup_{n\to\infty}\left|\mathrm{sTr}A^n\right|^{1/n}.\ee
On the other hand, it follows from the conditions on $A$ that
\eqref{str_series}, extended meromorphically to $\mathbb C$, has a pole on
$\{|z|=1/\rho(A)\}$, and no poles closer to the origin.
Hence $R=\rho(A)$, and the proof of the lemma is complete.
\epr

\begin{theorem} \label{hom_entropy_bound}
Let $\mathcal{C}$ be a saturated $A_\infty$-category over $\mathbb C$ and
$F:\mathcal C \to \mathcal C$ a functor.
Assume that the induced map on Hochschild homology, $HH_*(F)$, satisfies the
condition of Lemma \ref{str_radius}, then
\ben \log\rho(HH_*(F))\leq h_0(F). \een
\end{theorem}

\bpr Let $S$ denote the Serre functor of $\mc C$. By Lemma
\ref{complexity in Dbvec}, \be
\delta_0(k,R\Hom(F^N,S))=\sum_n\dim\mathrm{Ext}^n(F^N,S)\ee thus
\be \left|\chi(R\Hom(F^N,S))\right|\leq \delta_0(k,R\Hom(F^N,S)).
\ee Applying the triangle inequality,
\begin{align} \nonumber \delta_0(k,& R\Hom(F^N,S)) \\
&\leq \delta_0(k,R\Hom(G\otimes R\Hom(G,\_),S))
\delta_0(R\Hom(G\otimes R\Hom(G,\_),S),R\Hom(F^N,S)).
\end{align}
The first factor is independent of $N$, while the second is bounded above by
$\delta_0(G\otimes R\Hom(G,\_),F^N)$ as follows from the retraction map property
of the functor $R\Hom(\_,S):\mathrm{Fun}(\mc C,\mc C)\to D^b(k
\mhyphen\mathrm{Vect})^{\mathrm{op}}$.
Now,
\begin{align} \nonumber
\delta_0(G\otimes & R\Hom(G,\_),F^N) \\
&\leq \delta_0(G\otimes R\Hom(G,\_),F^N\circ(G\otimes R\Hom(G,\_)))
\delta_0(F^N\circ(G\otimes R\Hom(G,\_)),F^N).
\end{align}
By the retraction property of $F^N\circ\_$, the second factor is bounded above
by $\delta_0(G\otimes R\Hom(G,\_),\mathrm{id}_{\mc C})$ which is independent of
$N$ and finite by smoothness of $\mathcal C$.
The first factor is bounded above by $\delta_0(G,F^N G)$, hence
\be \limsup_{N\to\infty}
\left|\chi(R\Hom(F^N,S))\right|^{1/N} \leq \re^{h_0(F)}. \ee
By the Lefschetz fixed point theorem for Hochschild homology \cite{lunts2012} we
have \be \chi(R\Hom(F^N,S))=\mathrm{sTr}(HH_*(F^N)) \ee
and by Lemma \ref{str_radius},
\be \limsup_{N\to\infty}
\left|\mathrm{sTr}(HH_*(F)^N)\right|^{1/N}=\rho(HH_*(F)) \ee
which completes the proof.
\epr

\begin{remark}
We suspect that the conclusion of Theorem \ref{hom_entropy_bound} holds without
the condition on $F$.
\end{remark}

\subsection{Example: semisimple categories}

Fix a finite set $X$ and a field $k$. Let $D^b(X)$ denote the
triangulated dg-category of $X$-indexed families of bounded complexes over $k$.
Any (additive, graded) endofunctor of $D^b(X)$ is isomorphic to one of
the form
\begin{equation}
\Phi_K(V) = (p_1)_*(K\otimes p_2^*V)
\end{equation}
for some $K\in D^b(X\times X)$ with trivial differential.
Such a $K$ corresponds to a matrix of Poincar\'e (Laurent--) polynomials
\begin{equation}
P_K(z)=\left(\sum_{n\in\mathbb Z}\dim(K^n_{x,y})z^n\right)_{x,y\in X}
\end{equation}
with non-negative integer coefficients.
Note that composition of functors corresponds to multiplication of
matrices.

Let $G\in D^b(X)$ be the generator with $G_x=k$ for $x\in X$.
Then
\begin{equation}
\sum_{n\in\mathbb Z}\dim\mathrm{Ext}^n(G,\Phi_K^N(G))e^{-nt}=\|P_K^N(e^{-t})\|_1
\end{equation}
where $\|A\|_1=\sum|a_{ij}|$ is the matrix $1$-norm and we use non-negativity of
$P_K^N(e^{-t})$.
We obtain
\begin{equation}
h_t(\Phi_K) = \log\rho(P_K(e^{-t}))
\end{equation}
by Gelfand's formula for the spectral radius
\begin{equation}
\rho(A) = \lim_{n\to\infty}\|A^n\|^{\frac{1}{n}}
\end{equation}
which holds for any matrix norm.

We can say a bit more about the function $\rho(P_K(x))$.
By Perron--Frobenius theory, the spectral radius of a non-negative matrix is an
eigenvalue.
Hence, if we consider the spectral curve
\begin{equation}
C=\{(x,\lambda)\in\mathbb R_{>0}\times\mathbb R\,|\,\det(P_K(x)-\lambda)=0\}
\end{equation}
then the graph of $\rho(P_K(x)):\mathbb R_{>0}\to\mathbb R$ is the maximal
branch of $C$.

\subsection{Regular maps}

Let $X$ be a smooth projective variety over $\mathbb C$, $f:X\to X$ a regular
map, and $f^*:D^b(X)\to D^b(X)$ the pullback functor.
We will show that, under the generic condition of Lemma \ref{str_radius}, the
entropy of $f^*$ is constant and equal to the logarithm of the spectral radius
of $H^*(f;\mathbb Q)$, the induced map on cohomology.
By a result of S.~Friedland \cite{friedland}, $\log\rho(H^*(f;\mathbb Q))$ is
also equal to the topological entropy of $f$, more generally for compact
K\"ahler manifolds.

We begin with a lemma which gives a sufficient condition for the entropy
$h_t(F)$ to be constant in $t$.

\begin{lemma} \label{constEntropy}
Let $\mathcal C$ be a saturated $A_{\infty}$-category and $F$ an endofunctor of
$\mathcal C$.
Suppose there exists a generator $G$ of $\mc C$ and $M\ge 0$ such that
\begin{equation*}
\mathrm{Ext}^n(G,F^NG)=0\quad \text{ for }|n|>M,N\ge 0
\end{equation*}
then $h_t(F)$ is a constant function.
\end{lemma}

\begin{proof}
Use
\begin{equation}
\dim\mathrm{Ext}^n(G,F^NG)e^{-M|t|} \le \dim\mathrm{Ext}^n(G,F^NG)e^{-nt}\le
\dim\mathrm{Ext}^n(G,F^NG)e^{M|t|}
\end{equation}
to conclude $h_0(F) \le h_t(F) \le h_0(F)$.
\end{proof}

Note that Lemma \ref{constEntropy} applies in particular to functors
preserving a bounded t-structure of finite homological dimension, as the generator can be chosen to
lie in the heart of the t-structure.

\begin{theorem} \label{entropyRegular}
Let $X$ be a smooth projective variety over $\mathbb C$, $f:X\to X$ a regular
map, and $f^*:D^b(X)\to D^b(X)$ the pullback functor.
Suppose that the induced map on cohomology, $H^*(f;\mathbb Q)$, satisfies the
condition of Lemma~\ref{str_radius}.
Then $h_t(f^*)$ is constant and equal to $\log\rho(H^*(f;\mathbb Q))$.
\end{theorem}

\begin{proof}
Clearly, $f^*$ preserves $\mathrm{coh}(X)\subset D^b(X)$ which has finite
homological dimension, since $X$ is smooth and proper.
Hence, by Lemma~\ref{constEntropy}, $h_t(f^*)$ is constant in $t$.

We will show that
\begin{equation} \label{entropyRegUpper}
h_0(f^*)\leq\log\rho(H^*(f,\mathbb Q))
\end{equation}
without the assumption of Lemma \ref{str_radius}.
The claim then follows from Theorem~\ref{hom_entropy_bound} by the
Hochschild-Kostant-Rosenberg isomorphism and Hodge theory which give
\begin{equation}
HH_*(D^b(X))\cong H^*(X;\mathbb C)
\end{equation}
as $\mathbb Z/2$-graded vector spaces.

It remains to show \eqref{entropyRegUpper}.
Choose $L\in\mathrm{Pic}(X)$ very ample and consider
\begin{equation}
G=\bigoplus_{k=1}^{d+1}L^k
\end{equation}
where $d=\dim X$.
Both $G$ and $G^*$ are generators of $D^b(X)$ by \cite{orlovGenerators}.
Since $X$ is projective, $f$ preserves the nef cone in $N^1(X)_{\mathbb R}$.
In particular, $(f^*)^NG^*$ is a sum of anti-nef line bundles for every $N\ge
0$.
Note that $G^*$ itself is a sum of anti-ample line bundles, hence, using
$\mathrm{nef}+\mathrm{ample}=\mathrm{ample}$, we conclude that
\begin{equation}
G^*\otimes(f^*)^NG^*=\bigoplus_{1\leq k,l\leq d+1}L^{-k}\otimes (f^*)^NL^{-l}
\end{equation}
is a sum of anti-ample line bundles.
Thus, by the Kodaira vanishing theorem, $\mathrm{Ext}^*(G,(f^*)^NG^*)$ is
concentrated in degree $d$.
We get
\begin{align}
\dim\mathrm{Ext}^*(G,(f^*)^NG^*)&=
\left|\chi\mathrm{Ext}^*(G,(f^*)^NG^*)\right|\\
 &=\left|\int\mathrm{ch}(G^*)
 (f^*)^N\mathrm{ch}(G^*)\mathrm{td}(X)\right|
\end{align}
where $f^*$ is the induced map on $H^*(X;\mathbb Q)$.
It follows that
\begin{equation}
\lim_{N\to\infty}\sqrt[N]{\dim\mathrm{Ext}^*(G,(f^*)^NG^*)}\leq
\rho(H^*(X;\mathbb Q))
\end{equation}
which shows the claim, since the left hand side is $\exp(h_0(f^*))$ by
Theorem~\ref{entropyPoincareSeries}.
More precisely, we use here a variant of this theorem, with the same proof,
where one is allowed to take two different generators.
\end{proof}

\subsection{Entropy of the Serre functor}

The notion of a Serre functor was introduced by Bondal and Kapranov in
\cite{BK}. If it exists, it is unique up to natural isomorphism,
and thus its entropy an invariant of the triangulated category.
In the examples below, the entropy of the Serre functor will always turn out to
be of the linear form $at+b$, with $a$ having some interpretation as a
dimension. We note however, that this can fail in general, for example for
categories which are orthogonal sums.

\subsubsection{Fractional Calabi-Yau categories}

A triangulated category $\mc T$ with Serre functor $S$ is called
\emph{fractional Calabi-Yau} of dimension $\frac{m}{n}\in\mathbb Q$ if
\begin{equation}
S^n\cong [m]
\end{equation}
c.f. \cite{keller2008calabi}.
Suppose that $\mc T$ is the homotopy category of a saturated
$A_{\infty}$-category $\mc C$.
Using
\begin{equation}
h_t(F^n)=nh_t(F)\quad\text{ for } n\ge 1, \qquad\qquad h_t([n])=nt\quad \text{
for } n\in\mathbb Z
\end{equation}
we see that
\begin{equation}
h_t(S)=\frac{m}{n}t
\end{equation}
in this case.

\subsubsection{Smooth projective varieties}

Let $X$ be a smooth projective variety over a field $k$.
The bounded derived category of coherent sheaves on $X$,
$D^b(X)=D^b(\mathrm{coh}(X))$ is saturated with Serre functor
\begin{equation}
S=\_\otimes\omega_X[\dim X]
\end{equation}
where $\omega_X$ is the canonical sheaf.

\begin{prop}
$h_t(S)=\dim(X)t$
\end{prop}

Note that in particular, if $\mc C$ is a saturated category and $h_t(S)$ is
not of the form $nt$ for some $n\in\mathbb Z_{\ge 0}$, then $\mc C$ cannot
come from an algebraic variety.

\begin{proof}
This is equivalent to $h_t(\_\otimes\omega_X)=0$, which is a special case of
Lemma \ref{entropyForPic} below.
\end{proof}

\begin{lemma} \label{entropyForPic}
Let $X$ be a smooth projective variety over a field $k$, $L\in\mathrm{Pic}(X)$.
Then $h_t(\_\otimes L)=0$.
\end{lemma}

\begin{proof}
By lemma \ref{constEntropy}, $h_t(\_\otimes L)$ is constant.
Let $G\in\mathrm{coh}(X)$ be a generator of $D^b(X)$.
It suffices to show that
\begin{equation}
\sum_{n\in\mathbb Z}\dim \mathrm{Ext}^n(G,G\otimes L^N) = O(N^d),\quad
N\to\infty
\end{equation}
for some $d$.
This holds by the standard fact that
\begin{equation}
\dim H^n(\mathcal{O}_X, \mathcal F\otimes L^N) = O(N^{\dim X}),\quad N\to\infty
\end{equation}
for any coherent sheaf $\mathcal F$, $n\ge 0$, see
\cite{lazarsfeld2004positivity}.
\end{proof}

\subsubsection{Quivers}

Let $Q$ be a quiver. We assume throughout that $Q$ is connected and does not
have oriented cycles.
Fix a field $k$ and form the path algebra $A=kQ$ of $Q$.
Let $\mathcal A$ be the category of finite dimensional left modules over $A$.
Its bounded derived category, $D^b\mathcal A$, is saturated with
Serre functor isomorphic to the derived functor of $\mathrm{Hom}(\_,A)^*$.

\begin{lemma}
If $X\in\mathcal A$ is projective, then $S(X)\in\mathcal A$ and $S(X)$ is
injective. If X has no projective summands, then $S(X)\in\mathcal A[1]$.
\end{lemma}

\begin{proof}
Note first that $\mathcal A$ is hereditary and thus any indecomposable object in
$D^b\mathcal A$ is contained in $\mathcal A[n]$ for some $n\in\mathbb
Z$.
Furthermore, if $M\in\mathcal A$ is indecomposable, then $S(M)$ is contained
either in $\mathcal A$ or in $\mathcal A[1]$.
Suppose $P$ is projective, then
\begin{equation}
\mathrm{Ext}^n(P,M)=\mathrm{Ext}^{-n}(M,S(P))=0\qquad\text{ for }n\neq 0,
M\in\mathcal A
\end{equation}
hence $S(P)\in\mathcal A$ and $S(P)$ is injective.
If $N\in\mathcal A$ is indecomposable and not projective, then there exists an
$M\in\mathcal A$ with
\begin{equation}
\mathrm{Ext}^1(N,M)=\mathrm{Ext}^{-1}(M,S(N))\neq 0
\end{equation}
thus $S(N)\notin\mathcal A$.
\end{proof}

Let $e_i$ be the idempotent corresponding to the vertex $i\in Q_0$.
Recall that $Ae_i$, $i\in Q_0$ is a complete list of indecomposable projectives
and $(e_iA)^*$, $i\in Q_0$ is a complete list of indecomposable injectives.

\begin{lemma}
If $Q$ is not Dynkin, then $\Phi^{-N}(P)\in\mathcal A$ for any projective
module $P$ and $N\ge 0$.
\end{lemma}

\begin{proof}
By the previous lemma, $\Phi^{-1}(M)\in\mathcal A$ for an
indecomposable $M\in\mathcal A$ as long as $M$ is not injective.
Suppose that $P$ is projective and $M=\Phi^{-N}(P)$ is injective.
Then $M$ is both preprojective and preinjective, thus $A$ of finite
representation type by a result of \cite{auslander1997representation} ,
contradicting the assumption that $Q$ is not Dynkin.
\end{proof}

If $Q$ is Dynkin, then $D^b\mathcal A$ is fractional Calabi-Yau, see
\cite{keller2008calabi}. We assume from now on that $Q$ is not Dynkin.
Then, by the lemma
\begin{equation}
\mathrm{Ext}^n(A,\Phi^{-N}(A))=0 \qquad \text{ for } N\ge 0, n\neq 0
\end{equation}
and using $\dim\mathrm{Hom}(A,M)=\dim M$ we get
\begin{equation}
h_t(\Phi^{-1})=\lim_{N\to\infty}\frac{1}{N}\log\dim\Phi^{-N}A.
\end{equation}
Furthermore, by duality,
\begin{equation}
\dim\mathrm{Ext}^n(A,\Phi^NA)=\dim\mathrm{Ext}^{1-n}(A,(\Phi^{-1})^{N-1}A)
\end{equation}
hence
\begin{equation}
h_t(S)=h_t(\Phi)+t=h_t(\Phi^{-1})+t.
\end{equation}

At this point, the problem of computing $h_t(\Phi^{-1})$ is one of linear
algebra, since $\dim M$ depends only on
\begin{equation}
M\in K_0(\mathcal A)\cong\mathbb Z^{Q_0}
\end{equation}
where a natural basis of $K_0(\mathcal A)$ is given by the simple modules.
Recall that the Euler form
\begin{equation}
\langle M,N\rangle =\sum_{n\in\mathbb Z}(-1)^n\dim\mathrm{Ext}^n(M,N)
\end{equation}
is a bilinear form on $K_0(\mathcal A)$ with matrix $E$ given by
\begin{equation}
E_{ij}=\delta_{ij}-n_{ij}
\end{equation}
where $n_{ij}$ is the number of arrows from the $j$-th to the $i$-th vertex.
Let $[S]$ denote the induced map on $K_0(\mathcal A)$.
By the defining property of the Serre functor, $E^{\top}=E[S]$, hence
\begin{equation}
[S]=E^{-1}E^{\top},\qquad [\Phi]=-E^{-1}E^{\top},\qquad
[\Phi^{-1}]=-E^{-\top}E.
\end{equation}
The linear map $[\Phi]$ is classically the Coxeter transformations for the
Cartan matrix $E+E^{\top}$.
Its spectral properties were investigated by Dlab and Ringel in
\cite{dlab1981eigenvalues}.
They find that $\rho=\rho([\Phi])=\rho([\Phi^{-1}])$ is an eigenvalue of
$[\Phi]$, $\rho\ge 1$, and $\rho=1$ if and only if $Q$ is of extended Dynkin
type.
Furthermore, if $P$ is projective, then
\begin{equation}
\lim_{N\to\infty}\left(\dim\Phi^{-N}P\right)^{\frac{1}{N}}=\rho.
\end{equation}
We summarize our results.

\begin{theorem}
Let $Q$ be a connected quiver without oriented cycles, not of Dynkin type, and
$S$ the Serre functor on $D^b(kQ)$. Then
\begin{equation}
h_t(S)=t+\log{\rho([S])}
\end{equation}
and $\rho([S])\ge 1$ with equality if and only if $Q$ is of extended Dynkin
type.
\end{theorem}

\subsection{Pseudo-Anosov maps}

A great insight of Thurston was that a ``typical'' element of the mapping
class group of a surface $M$ can be represented by a \emph{pseudo-Anosov map}
\cite{thurston, flp}, which is by definition a homeomorphisms $\phi:M\to M$ such
that there exist transverse measured foliations $(\mc F^s,\mu^s)$, $(\mc F^u,\mu^u)$ and
$\lambda>1$, such that
\be \phi(\mc F^s,\mu^s)=(\mc F^s,\frac{1}{\lambda}\mu^s),\quad \phi(\mc
F^u,\mu^u)=(\mc F^u,\lambda\mu^u).\ee
The \emph{stretch factor} $\lambda$ can be computed as follows.
Let $\mc S$ be the set of isotopy classes of simple closed curves on $M$ which
do not contract to a point.
For $\alpha,\beta\in\mc S$ the \emph{geometric intersection number}
$i(\alpha,\beta)$ is the minimum number of intersection points of simple closed
curves representing $\alpha$ and $\beta$.
Then for $\phi,\lambda$ as above and $\alpha,\beta\in\mc S$ one has
\be \label{intersectionGrowth}
\lim_{n\to\infty}\sqrt[n]{i(\alpha,\phi^n(\beta))}=\lambda. \ee

We restrict to the case when $M$ is closed, orientable, and of genus $g>1$.
The Fukaya category $\mathrm{Fuk}(M)$ is a $\mathbb Z/2$-graded
$A_\infty$-category over the Novikov field $\Lambda_{\mathbb R}$.
For our purposes, objects are embedded curves with orientation and the bounding
spin structure.
Let $D^{\pi}\mathrm{Fuk}(M)$ denote the triangulated hull, which
is locally proper (in the $\mathbb Z/2$-graded sense), and has a generator given
by a collection of loops \cite{seidel2011,efimov2012}.
In fact, an explicit resolution of the diagonal is constructed in
\cite{seidel2011,efimov2012}, proving homological smoothness.
Thus $D^{\pi}\mathrm{Fuk}(M)$ is saturated.
An element $\phi$ of the mapping class group of $M$ induces an autoequivalence
$\phi_*$ of $D^{\pi}\mathrm{Fuk}(M)$.

\begin{theorem} \label{pseudoAnosovEntropy}
Let $\phi$ be a pseudo-Anosov map with stretch factor $\lambda$, then
$h_0(\phi_*)=\log\lambda$.
\end{theorem}

It is a classical result, see \cite{flp}, that the topological entropy of a
pseudo-Anosov map $\phi$ is $\log\lambda$, and that $\phi$ minimizes topological
entropy in its isotopy class.

\bpr
Note first that Theorem \ref{entropyPoincareSeries} is also valid, with
essentially the same proof, for saturated $\mathbb Z/2$-graded
$A_\infty$-categories in the sense that
\be h_0(F)=\lim_{N\to\infty}\frac{1}{N}\log\left(\dim
\mathrm{Ext}^0(G,F^NG)+\dim\mathrm{Ext}^1(G,F^NG)\right)\ee
for any endofunctor $F$.
On the other hand, if $\alpha,\beta$ are simple closed loops in $M$ then
\be\dim\mathrm{Ext}^0(\alpha,\phi^N(\beta))+\dim\mathrm{Ext}^1(\alpha,\phi^N(\beta))=
i(\alpha,\phi^N(\beta)) \ee
for $N\gg 0$ by Lemma \ref{HFsurface} below.
Now, $D^{\pi}\mathrm{Fuk}(M)$ has a generator $G$ which is a direct sum of
$2g+1$ simple loops.
In view of \eqref{intersectionGrowth} the theorem follows.
\epr

The following is considered well-known to experts in Floer theory.

\begin{lemma} \label{HFsurface}
Let $M$ be a closed surface of positive genus with symplectic structure,
$\alpha$ and $\beta$ simple closed loops on $M$ which are not isotopic.
Then
\be \label{HFintersect} \dim HF^0(\alpha,\beta)+\dim
HF^1(\alpha,\beta)=i(\alpha,\beta)\ee
where $HF^k(\alpha,\beta)$ denotes Floer cohomology over the Novikov field.
\end{lemma}

\bpr
If $\alpha$ and $\beta$ are already in minimal position, i.e.
$i(\alpha,\beta)=|\alpha\cap\beta|$, then \eqref{HFintersect} is clear, since
there are no immersed bi-gons with boundary on $\alpha\cup\beta$, and thus the
Floer differential vanishes on $CF^*(\alpha,\beta)$ which has a basis given by
the intersection points.

In general, there is an embedded loop $\alpha'$ which is isotopic to $\alpha$
and intersects $\beta$ minimally.
Moreover, we can construct an isotopy from $\alpha$ to $\alpha'$ which is a
composition of Hamiltonian isotopies and isotopies which do not create or remove
any intersection points between $\alpha$ and $\beta$.
The invariance of $HF^*(\alpha,\beta)$ under the former is true by the general
theory, we will show that invariance also holds for the latter.

Let us observe first that near each intersection point $p$ of $\alpha$ and
$\beta$ there is a chart taking $\alpha$ to the x-axis and $\beta$ to the
y-axis. In such a chart, any bi-gon contributing to $d[p]$ must map locally to
either the second or fourth quadrant, see \cite{abouzaid2008}.
Furthermore, the assumption on $\alpha$ and $\beta$ ensures that bi-gons
contributing to $d^0$ (resp. $d^1$) cannot form a closed chain.

Let $M$ be the matrix representing the Floer differential $d^0$ or
$d^1$, with respect to the basis given by intersection points.
Since $\alpha$ and $\beta$ are not isotopic, there can be at most
one bi-gon contributing to any entry of $M$, hence the position of
the non-zero entries does not change under the isotopy. The
observations made above translate into the following properties of
$M$.
\begin{enumerate}
  \item Any column or row of $M$ has at most two non-zero entries.
  \item Let $Q$ be the quiver with vertices the columns and rows of $M$ and an
  arrow from $i$ to $j$ if the entry in the $i$-th column and $j$-th row of $M$
  is non-zero. Then the underlying undirected graph of $Q$ has no loops.
\end{enumerate}
One shows that the rank of a matrix satisfying the second property is
determined by the positions of the non-zero entries alone, and not the specific
values. 
We conclude that $\dim HF^*(\alpha,\beta)$ is invariant under isotopy.
\epr

\section{Density of phases} \label{Kronecker pairs}

\subsection{Preliminaries}

\subsubsection{On Bridgeland stability conditions}

We start by  recalling  some definitions and results by
Bridgeland. Here we  introduce the notations $\HH^{\mc A}$(see
Definition \ref{H to A} and Remark \ref{HN property and locally
finiteness}) and  $P_\sigma$ (the set of semistable phases of a
stability condition $\sigma$), which are  useful later.

In \cite{Bridg1} a stability conditions on a triangulated category
$\mc T$  is defined as a pair $\sigma=( \mc P, Z )$ satisfying certain
 axioms, where the first component is   a family of  full
additive subcategories  $\{\mc P(t)\}_{t \in \RR}$ (by the axioms
it follows that they are abelian) and the second component is  a
group homomorphism $\bd K_0(\mc T) &\rTo^Z & \CC\ed$. The non-zero objects in $\mc P(t)$ are said to be $\sigma$-semistable of
phase $t$ and   the phase $\phi_\sigma(E)$ of a semistable  $E\in \mc P(t)$   is well defined by $\phi_\sigma(E)=t$.

 To give a
description of the Bridgeland stability conditions on $\mc T$,  we recall first  the
following definition:

\begin{df} \label{Z-semistable} Let $(\mc A,\bd K_0(\mc A)& \rTo^{Z}& \CC \ed ) $ be an  abelian category  and a stability
function on it \footnote{I.e. $Z$ is homomorphism, s. t. $Z(X)\in \HH=\{ r \exp(\ri \pi t)\vert r>0 \ \ \mbox{and} \ \ 0< t \leq 1 \}$ for $X \in \mc A$, $X\neq 0$}. A non-zero object $X\in \mc A $ is said to be $Z$-semistable   of phase $t$ if    every $\mc A$-monic $X' \rightarrow X$   satisfies\footnote{For $u\in \HH$ we denote by $\arg(u)$ the unique number satisfying $\arg(u)\in (0,1]$, $u=\exp(\ri \pi \arg(u))$. It is convenient to set $\arg(0)=-\infty $.} $ \arg Z(X')\leq \arg Z(X)=\pi t$ (if equality is attained only for $X'\cong X$ then $X$ is said to be stable).
\end{df}

 If  $\mc A \subset \mc T$ is a bounded $t$-structure and     $Z: K_0(\mc A) \rightarrow \CC $ is a stability function, which satisfies the HN property (\cite[Definition 2.3]{Bridg1})  then   the pair $\sigma=(\mc P, Z_e)$ defined by(see the proof of \cite[Proposition 5.3]{Bridg1}):
\begin{itemize}
    \item[\textbf{a)}] If $t\in (0,1]$, define ${\mc P}(t)\subset \mc A$ as the set of $Z$-semistable objects  of phase $t$ in
    $\mc A$(as defined in definition \ref{Z-semistable}). If $t\in (n,n+1]$, $n\in \ZZ$,  define ${\mc P}(t)=T^n({\mc P}(t-n))$.
    \item[\textbf{b)}] Define $\bd K_0(\mc T)& \rTo^{Z_e}& \CC \ed$, such  that:
    $\bd K_0(\mc A) & \rTo^{K_0(\mc A \subset \mc T)} &  K_0(\mc T)& \rTo^{Z_e}& \CC \ed$ $=  \bd K_0(\mc A)& \rTo^{Z}& \CC \ed$
\end{itemize}
is a stability condition on $\mc T$ (furthermore all stability conditions on $\mc T$ are obtained by this procedure). Let us denote:
\begin{df}\label{H to A} Let  $\mc A \subset \mc T$ be a  bounded t-structure  in a triangulated category $\mc T$.  We   denote by $\HH^{\mc A}$
the   family   of stability conditions in $\mc T$ obtained by \textbf{a)}, \textbf{b)} above varying $Z$ in the set of all  stability
functions on $\mc A$  with  HN property.
 In particular $  \HH^{\mc A} \ni (\mc P,Z) \mapsto Z_{\vert K_0(\mc A)}$ is a  bijection between $\HH^{\mc A}$ and this set.
\end{df}

 Bridgeland  proved that the set of all
stability conditions, satisfying a property called locally finiteness,  on a triangulated category $\mc T$  is   a complex manifold, denoted by $\st(\mc T)$.

\begin{remark} \label{HN property and locally finiteness} Let $\mc A\subset \mc T$ be as in the previous definition.  If $\mc A$ is an abelian category
of finite length, then any  stability function  $Z: K_0(\mc A) \rightarrow \CC $  satisfies the HN property (\cite[Proposition 2.4]{Bridg1}). If in addition
  $\mc A$ has finitely many, say $s_1,s_2,\dots,s_n$, simple objects then all   stability conditions in  $\HH^{\mc A}$ are locally finite.
   Whence, in this setting we have  $\HH^{\mc A} \subset \st(\mc T)$ and bijection  $  \HH^{\mc A} \ni (\mc P,Z) \mapsto (Z(s_1),\dots,Z(s_n)) \in \HH^n$.
\end{remark}
Finally, we introduce the notation:
\begin{df} \label{P_sigma} Let  $\mc T$ be a triangulated category and  $\sigma =(\mc P, Z) \in \st(\mc T)$ a stability condition on it.  We
denote:\footnote{When the triangulated category $\mc T$ is  fixed in advance we write just  $P_\sigma$.}
\be P^{\mc T}_\sigma = \exp(\ri \pi \{t \in \RR \vert {\mc P}(t) \neq \{ 0 \} \})\subset S^1.\ee
By ${\mc P}(t+1)={\mc P}(t)[1]$\footnote{which is one of the Bridgeland's axioms} it follows $-P^{\mc T}_\sigma=P^{\mc T}_\sigma$.
\end{df}




\subsubsection{On $\theta$-stability and  a theorem by King} \label{stability in abeilan}  In the next subsection  we use a result by King. We recall first

\begin{df}[$\theta$-stability] \label{theta stability} Let $\theta:K_0(\mc A) \rightarrow \RR $ be a non-trivial group homomorphism, where $\mc A$ is an
 abelian category. Then $X \in \mc A$ is called $\theta$-semistable if $\theta(X)=0$ and for each monic arrow $X'\rightarrow X$ in $\mc A$ we have $\theta(X')\geq 0$ (if $\theta(X')=0$ only for the sub-objects $0$ and $X$ then it is called $\theta$-stable).
 \end{df}
 \begin{remark} $Z$-semistable of phase $t$ (as defined in Definition \ref{Z-semistable}) is the same as $\theta$-semistable with $\theta=-Im(\re^{-\ri \pi t} Z)$.  \end{remark}
 From Proposition 4.4 in \cite{King} it follows
 \begin{prop}[A. King]  \label{prop king} Let $A$ be a finite dimensional, hereditary $\CC$-algebra. Let $\alpha \in K_0(A\mbox{-}{\rm Mod})$.
 Then the following conditions are equivalent:

\begin{enumerate}
    \item There exist $X \in A\mbox{-}{\rm Mod}$  and  a non-trivial $\theta:K_0(A\mbox{-}{\rm Mod}) \rightarrow \RR$, s. t. $[X]=\alpha$ and $X$ is $\theta$-stable.
    \item $\alpha$ is a Schur root, which by definition means that  some $Y\in A\mbox{-}{\rm Mod} $ with $[Y]=\alpha$  satisfies ${\rm End}_{A\mbox{-}{\rm Mod}}(Y)=\CC$.
\end{enumerate}
  \end{prop}
This Proposition will be used in the proof of Corollary \ref{density for kronecker0}.

\subsection{Dynkin, Euclidean quivers and Kronecker quiver} \label{Euc and Kron}
In this subsection we comment the set $P_\sigma$ as $\sigma$ varies in the set of stability conditions on Dynkin, Euclidean quivers and on the Kronecker quiver. The main results here are Lemma \ref{Dynkin},
Corollary \ref{coro for density1} and Corollary \ref{density for kronecker}.

\subsubsection{Quivers and Kac's theorem} \label{quivers and Kac}
For any quiver $Q$ we  denote  its set of vertices  by $V(Q)$, its set of arrows by $Arr(Q)$   and the underlying non-oriented graph by $\Gamma(Q)$. Let
\be \label{origin,end} Arr(Q) \rightarrow  V(Q)\times V(Q)  \ \  \ \ a \mapsto (s(a),t(a)) \in V(Q) \times V(Q) \ee  be the   function assigning
 to an arrow $a\in Arr(Q)$ its origin $s(a) \in V(Q)$ and its end $t(a) \in V(Q)$.  A vertex $v \in V(Q)$ is called \textit{source/sink} if all arrows touching
 it start/end at it (more precisely $v \neq t(a)$/$v \neq s(a)$ for each $a\in Arr(Q)$).

Throughout  this section \ref{Kronecker pairs} the term
\textit{Dynkin quiver} means a quiver $Q$,  s. t. $\Gamma(Q)$ is
one of the simply laced Dynkin diagrams  $A_m$,$m\geq 1$, $D_m$,
$m \geq 4$,  $E_6$,  $E_7$,  $E_8$ (see for example \cite[p.
32]{Barot}) and the term \textit{Euclidean quiver} means an
acyclic quiver $Q$,  s. t. $\Gamma(Q)$ is one of the extended
Dynkin diagrams  $\wt{A}_m$,$m\geq 1$, $\wt{D}_m$, $m \geq 4$,
$\wt{E}_6$,  $\wt{E}_7$,  $\wt{E}_8$ (see for example   \cite[fig.
(4.13)]{Barot}).  By $K(l)$, $l\geq 1$  we denote the quiver, which consists of two vertices with $l$
parallel arrows between them. Note that $K(1)$ is Dynkin, $K(2)$ is Euclidean. We call $K(l)$, $l\geq 3$ and
$l$-Kronecker quiver.

 Recall the  Kac's Theorem.
\begin{remark}[On Kac's Theorem] Let $Q$ be a connected quiver  without edges-loops.
 In \cite{Kac} is defined the positive root system of $Q$. We denote this root system by $\Delta_+(Q)\subset \NN^{V(Q)}$.
 For $X \in Rep_{\CC}(Q)$ we denote by $\ul{\dim}(X)\in \NN^{V(Q)}$ its dimension vector.  The main result of \cite{Kac} (we consider only the field $\CC$) is:
\be \label{Kac theorem} \{\ul{\dim}(X) \vert X \in Rep_{\CC}(Q), X \ \mbox{is indecomposable}\} =  \Delta_+(Q). \ee
\end{remark}

The Euler form of any quiver $Q$ is defined by
\be \label{euler form} \scal{\alpha,\beta}_Q=\sum_{j\in V(Q)} \alpha_j \beta_j - \sum_{j\in Q_1} \alpha_{s(j)} \beta_{t(j)}, \qquad \alpha, \beta \in \NN^{V(Q)}.\ee

The set  $\Delta_+(Q)$ has a simple description for Dynkin, extended Dynkin or hyperbolic quivers ($K(l)$, $l\geq 3$ are hyperbolic quivers) as shown by Kac in
\cite{Kac}.  It is determined by the Euler form as follows
\be \label{roots by euler} \Delta(Q)=\{r\in\NN^{V(Q)}\setminus \{0\}\vert \scal{r,r}_Q\leq 1\}, \qquad \Delta_+(Q)=\Delta(Q)\cap \NN^{V(Q)}.\ee

\ul{If $Q$ has no  oriented cycles}, (then $Q$ is called acyclic and the path algebra $\CC Q$ is finite dimensional) in addition  to this we have an isomorphism
 $K_0(Rep_\CC(Q))\cong \ZZ^{V(Q)}$ determined by  $K_0(Rep_\CC(Q)) \ni [X]$ $\mapsto $ $\ul{\dim}(X) \in \ZZ^{V(Q)}$, where $X \in Rep_\CC(Q)$. In particular, for any homomorphism $Z:K_0(Rep_\CC(Q)) \rightarrow \CC$ and any $X \in Rep_\CC(Q)$ we have
\be \label{Z(X) for quivers} Z(X)= \sum_{i\in V(Q)} \ul{\dim}_i(X) \  Z(s_i)=(v,\ul{\dim}(X)), \qquad \{ v_i= Z(s_i) \}_{i\in V(Q)},\ee
 where $s_i$ is the simple representation with $\CC$ in the vertex $i\in V(Q)$ and $0$ in the other vertices.
 \textit{Throughout this section \ref{Kronecker pairs}   $(,)$ denotes the bilinear form  on  $\CC^{V(Q)} \times \CC^{V(Q)}$
  defined by  $(\alpha,\beta)=\sum_{i\in V(Q)} \alpha_i \beta_i $, $\alpha,\beta \in \CC^{V(Q)}$, NOT the symmetrization
  $\scal{\alpha,\beta}_Q+\scal{\beta,\alpha}_Q$ of $\scal{,}_Q$. We mention once this symmetrization and denote it by $(,)_Q$.}

\subsubsection{The inclusion \texorpdfstring{$P_\sigma \subseteq R_{v,\Delta_+}$}{\space}.}
\begin{lemma} Let $\mc T$ be any triangulated category. Then for each $\sigma =(\mc P, Z) \in \st(\mc T)$ we have:
$ \{t \in \RR \vert {\mc P}(t) \neq 0 \} = \{ \phi_{\sigma}(I) \vert I \ \mbox{is} \  \mc T \mbox{-indecomposable} \  \ \mbox{and } \ \sigma\mbox{-semistable} \}$.
\end{lemma}
\bpr Let $X \in \mc P(t)$ be non-zero. Since $\mc P(t)$ is  of finite length then we have a decomposition in  $\mc P(t)$ of the form
$X \cong \bigoplus_{i=1}^n X_i$, where  $X_i$ are indecomposable in ${\mc P}(t)$, therefore (here we use that $\mc P(t)$ is abelian)
there are not non-trivial idempotents in ${\rm End}_{{\mc P}(t)}(X_i)={\rm End}_{\mc T}(X_i)$, hence $X_i$ is indecomposable in $\mc T$.
Whence, we see that $t=\phi_{\sigma}(X_i)$, where  $X_i$ is an indecomposable in $\mc T$ and $\sigma$-semistable.  The lemma follows.
\epr

\begin{coro} \label{coro P_sigma subset 1}  Let $\mc A$ be a hereditary abelian category. For each $\sigma =(\mc P, Z) \in \st(D^b(\mc A))$ holds the inclusion:
\be \label{P_sigma subset 1} P_\sigma  \subseteq \left \{\pm \frac{ Z(X)}{\abs{Z(X)}}\left \vert X \  \mbox{is indecomposable in } \ {\mc A}, Z(X)\neq 0 \right. \right \}. \ee
\end{coro}
\bpr  Take any $t\in  \RR $ with $\mc P(t) \neq \{0\}$. From the previous lemma there is  a semi-stable, indecomposable  $X\in   D^b(\mc A)$, s. t. $\phi_{\sigma}(X)=t$.
Since $\mc A$ is hereditary then $X=X'[i]$ for some indecomposable $X'\in {\mc A}$, $i \in \ZZ$. Now we can write
\ben (-1)^i Z(X')=Z(X)=m(X) \exp(\ri \pi \phi_\sigma(X))=m(X) \exp(\ri \pi t)  \qquad m(X)>0,\een
where we use that $X$ is $\sigma$-semistable and one of the Bridgeland's axioms  (\cite[Definition 1.1 a)]{Bridg1}). The corollary is proved.
\epr
When ${\mc A}=Rep_{\CC}(Q)$ with $Q$-acyclic  we can rewrite this corollary in a useful form.
 Putting  \eqref{Kac theorem} and \eqref{Z(X) for quivers} in the righthand side of \eqref{P_sigma subset 1} with $\mc A = Rep_\CC(Q)$
 we get a set  $\left \{\pm \frac{ (v,r)}{ \abs{ ( v,r ) }}\left \vert r\in \Delta_+(Q), (v,r)\neq 0 \right. \right \}$,
 where $v\in \CC^{V(Q)}$ is a non-zero vector. It is useful to define
 \begin{df} For any finite set $F$, any  subset $A\subset \NN^{F}\setminus \{0\}$ and any non-zero vector $v\in \CC^{F}$ we denote
 \footnote{When the set $F$ is clear  we write just  $R_{v,A}$.}
 \be\label{R_v,A}  R^{F}_{{v},A}= \left \{\pm \frac{ ({v},{r})}{ \abs{ ( {v},{r} ) }}\left \vert {r}\in A, ({v},{r})\neq 0 \right. \right \} \subset S^1, \
  \mbox{where} \ ({v},{r})=\sum_{i\in F} v_i \ r_i. \ee
 \end{df}
 Then we can rewrite \eqref{P_sigma subset 1} as follows (we assume that $Q$ is an acyclic, because we used \eqref{Z(X) for quivers}, which holds only
 for acyclic quivers) :
 \begin{coro} \label{coro P_sigma subset 2}  Let $Q$ be an acyclic quiver. For any  $\sigma =(\mc P, Z) \in$  $  \st(D^b(Rep_\CC(Q)))$  holds the inclusion
\be \label{P_sigma subset 2} P_\sigma  \subseteq R_{{v},\Delta_+(Q)} \qquad v=\{ {v}_i= Z(s_i) \}_{i\in V(Q)}. \ee
\end{coro}

\subsubsection{On the set  \texorpdfstring{$R_{{v},\Delta_+(Q)}$}{\space}}

\begin{lemma} \label{Dynkin}
 Let $Q$ be a Dynkin quiver.  For any stability condition $\sigma \in \st(D^b(Rep_\CC(Q)))$ the set of semi-stable phases $P_\sigma$ is  finite.
\end{lemma}
\bpr It is well known that for a Dynkin quiver $Q$ the positive root system $\Delta_+(Q)$ is finite. Hence  for any non-zero ${v} \in \CC^{V(Q)}$
the set $R_{{v},\Delta_+(Q)}$ is finite. Now the lemma follows from Corollary \ref{coro P_sigma subset 2}.
\epr

\begin{lemma} \label{finite number sonv seq} Let $Q$ be an Euclidean quiver (see subsubsection \ref{quivers and Kac} for definition).
For any non-zero ${v} \in \CC^{V(Q)}$  the set $R_{{v},\Delta_+(Q)}$ is either finite or there exist $m\in \NN$, $p \in S^1$ and sequences
$\{  p^i_j  \subset S^1 \}_{i=1,\dots,m ; j\in\NN}$, s. t. $\{ \lim_{j\rightarrow \infty} p^i_j= p \}_{i=1}^m $ and
 $ R_{{v},\Delta_+}=\cup_{i=1}^m \{\pm p^i_j\}_{j\in\NN}$.
\end{lemma}
\bpr
 The root system $\Delta$ of an Euclidean quiver  $Q$ (as described in the first equality of \eqref{roots by euler}) has an element
 ${\delta} \in \NN_{\geq 1}^{V(Q)}$ with the properties $\Delta\cup \{0\} + \ZZ {\delta} \subset \Delta\cup \{0\}$ and
 $\Delta \cup \{0\}/\ZZ {\delta}$ is finite (see \cite[p. 18]{WCB2}). Hence there is  a finite set $\{{\alpha}_1,{\alpha}_2\dots,{\alpha}_{m} \} \subset \Delta $,
  s. t.  $ \Delta \cup \{0\}=\bigcup_{i=1}^m ({\alpha}_i + \ZZ {\delta} )$. If for any $i\in \{1,2,\dots,m\}$ we choose the minimal
  $n_i\in \ZZ$, s. t. ${\alpha}_i + n_i {\delta} \in \Delta_+$ and denote ${\beta}_i={\alpha}_i+n_i {\delta}$, then
  $ \Delta_+=\bigcup_{i=1}^m ({\beta}_i + \NN {\delta} )$.  From the definition  \eqref{R_v,A}
of $R_{{v},\Delta_+}$ we see that
\be R_{{v},\Delta_+} = \bigcup_{i=1}^m \left \{\pm \frac{({v},{\beta}_i)+n ({v},{\delta}) }{\abs{({v},{\beta}_i)+n ({v},{\delta})}} \vert n\in \NN, i=1,2,\dots, m , ({v},{\beta}_i)+n ({v},{\delta})  \neq 0 \right \}.\ee If $({v},{\delta})=0$, then the set is finite. Otherwise   for $i= \{1,2,\dots,m\}$ we have
$ \lim_{n\rightarrow \infty} \frac{({v},{\beta}_i)+n ({v},{\delta}) }{\abs{({v},{\beta}_i)+n ({v},{\delta})}} =\frac{({v},{\delta}) }{ \abs{({v},{\delta})} }. $
\epr
From this lemma and Corollary \ref{coro P_sigma subset 2} it follows:
\begin{coro}\label{coro for density1}  Let $Q$ be an Euclidean quiver.
Then for any $\sigma \in D^b(Rep_k(Q))$ the set $P_{\sigma}$ is
either finite or has exactly  two limit points of the type
$\{p,-p\}$\footnote{In \cite{DK}  are shown examples of both the
cases: $P_{\sigma}$ is  finite and $P_{\sigma}$ is  with two limit
points. }.
\end{coro}
\bpr  If  $P_\sigma$ is infinite, then by the previous lemma and $P_{\sigma}\subset R_{{v},\Delta_+}$,  $\{ {v}_i=Z(s_i) \}_{i \in V(Q)}$ it follows that
$ R_{{v},\Delta_+}=\cup_{i=1}^m \{\pm p^i_j\}_{j\in\NN}$ with $\lim_{j\rightarrow \infty} p^i_j= p$ for $i=1,2,\dots, m$. In particular $P_\sigma$ can not
have more than two limit points.   Since $P_{\sigma}$ is infinite  then $P_{\sigma} \cap \{ p^i_j\}_{j\in\NN}$, $P_{\sigma} \cap \{- p^i_j\}_{j\in\NN}$
are infinite sets for some $i$ (recall that $-P_{\sigma}=P_{\sigma}$). Hence $\{p,-p\}$ are limit points of $P_{\sigma}$. The corollary follows.
\epr
Next we discuss the set $R_{{v},\Delta_+}$ for the $l$-Kronecker quiver $K(l)$ (two vertices with $l$ parallel arrows between them), $l\geq 3$.
In this case the vertices are two, so that ${v}$ has two complex coordinates. I.e. $R_{{v},\Delta_+}$ consists of fractions like
$\frac{n z_1 + m z_2}{\abs{n z_1 + m z_2}}$, where $z_1,z_2 \in \CC$, $n,m \in \NN$.   It is useful to note
\begin{remark} \label{remark for z_1,z_2}  Let $ z_i= r_i \exp(\ri \phi_i)$, $r_i>0$, $i=1,2$, $0 < \phi_2 < \phi_1 \leq \pi$.
Then \begin{gather} \label{arg(alpha...)} \frac{\alpha z_1 + \beta z_2}{\abs{\alpha z_1 + \beta z_2}}=\left \{  \begin{array}{ c  c } \exp\left (\ri f\left (\frac{\alpha}{\beta} \right) \right ) &   \alpha \geq 0, \beta > 0,  \\
\exp(\ri \phi_1) & \alpha > 0, \beta = 0, \end{array}  \right .\end{gather}
where $f:[0,\infty)\rightarrow [\phi_2, \phi_1 )\subset  (0, \pi )$ is the strictly increasing smooth function:
\begin{gather} f(x)=\arccos\left( \frac{x r_1 \cos(\phi_1)+r_2 \cos(\phi_2)}{\sqrt{x^2 r_1^2 + r_2^2 + 2 x r_1 r_2 \cos(\phi_1-\phi_2)}} \right ), \ \ f(0)=\phi_2, \ \lim_{x \rightarrow \infty }f(x) = \phi_1.\end{gather}

\end{remark}
From \eqref{euler form} we see that the Euler form for the quiver $K(l)$ is $ \scal{(\alpha_1, \alpha_2),(\beta_1, \beta_2)}_{K(l)}=\alpha_1  \beta_1 + \alpha_2  \beta_2 - l \alpha_1 \beta_2. $ Hence the positive roots are
\be \label{roots for K(l)} \Delta_{l+}=\Delta_+(K(l)) = \{(n,m)\in \NN^2  \vert n^2 + m^2 - l m n \leq 1 \}\setminus \{(0,0)\}.\ee
\begin{remark} Since the root systems $\Delta(K(l))$ with $l\geq 3$ will play an important role we reserve for them the notation $\Delta_{l}=\Delta(K(l))$, respectively $\Delta_{l+}=\Delta_{+}(K(l))$.   \end{remark}
The roots with $n^2 + m^2 - l m n = 1$ are called real roots and with  $n^2 + m^2 - l m n \leq 0$ - imaginary roots. We can represent the real and the imaginary roots as follows:
\begin{gather}\label{real roots} \Delta_{l+}^{re} = \{(1,0)\} \cup \{(0,1)\} \cup \left \{\frac{n}{m} =\ \frac{1}{2}\left (l\pm \sqrt{l^2-4 + \frac{4}{m^2}} \right ) \vert n, m \in \NN_{\geq 1},  (n,m)=1\right \} \\ \label{imaginary roots} \Delta_{l+}^{im} = \left \{\frac{1}{2}\left (l- \sqrt{l^2-4 } \right ) \leq \frac{n}{m} \leq \frac{1}{2}\left (l- \sqrt{l^2-4 } \right ) \vert n\in \NN_{\geq 0}, m \in \NN_{\geq 1}\right \}.\end{gather}

\begin{lemma}\label{lemma for R of a quiver}   Let ${v}=(z_1,z_2)$, $z_i= r_i \exp(\ri \phi_i)$,  $r_i>0$,  $0 < \phi_2 < \phi_1 \leq \pi$.  Let us denote $u= f \left (\frac{1}{2}\left (l- \sqrt{l^2-4 } \right )\right) $, $v=f\left (\frac{1}{2}\left (l+ \sqrt{l^2-4 } \right ) \right )$, where $f$ is defined in Remark \ref{remark for z_1,z_2}. Then $$ R_{{v},\Delta_{l+}}=\{\pm c_j\}_{j\in\NN} \cup \pm D \cup \{\pm a_j\}_{j\in\NN}, $$ where $D$ is a dense subset in the arc  $\exp (\ri [ u , v]) \subset S^1$, $\{a_j\}_{j\in\NN}$ is a sequence with $a_0= \exp(\ri \phi_2)$ and anti-clockwise monotonically converges to $\exp(\ri u)$, $\{c_j\}_{j\in\NN}$ is a sequence with $c_0= \exp(\ri \phi_1)$ and clockwise monotonically converging to $\exp(\ri v)$.
\end{lemma}
\bpr Now $R_{{v},\Delta_{l+}}=\left \{ \pm \frac{n z_1 + m z_2}{\abs{n z_1 + m z_2}} \vert (n,m) \in \Delta_{l+} \right \}$. We have a disjoint union
$\Delta_{l+}=\Delta_{l+}^{re}\cup \Delta_{l+}^{im}$, where $\Delta_{l+}^{re}$, $\Delta_{l+}^{im}$ are taken from \eqref{real roots}, \eqref{imaginary roots}.
Recall also that if $m\geq 1$ then $\frac{n z_1 + m z_2}{\abs{n z_1 + m z_2}}=\exp(\ri f\left ( n/m\right ) )$ (see Remark \ref{remark for z_1,z_2}).
Therefore we can write for $R_{{v},\Delta_{l+}}$:
\begin{gather}\left \{ \pm \frac{n z_1 + m z_2}{\abs{n z_1 + m z_2}} \vert (n,m) \in \Delta_{l+}^{re} \right \} \cup \left \{ \pm \frac{n z_1 + m z_2}{\abs{n z_1 + m z_2}}
 \vert (n,m) \in \Delta_{l+}^{im} \right \}= \nonumber \\
\{\pm\exp(\ri \phi_1)\} \cup \{\pm\exp(\ri \phi_2)\} \cup \left \{\pm\exp \left ( \ri f\left ( \frac{1}{2}\left (l\pm \sqrt{l^2-4 + \frac{4}{m^2}} \right ) \right )
 \right )\vert m\in \NN_{\geq 1} \right \} \nonumber \\
\cup \left \{\pm\exp \left ( \ri f\left (n/m \right ) \right )\vert n/m\in [u,v] \right \}.\nonumber \end{gather}
Now the lemma follows from the properties of $f$ given in Remark \ref{remark for z_1,z_2} and the fact that $\QQ \cap [u,v]$ is dense in $[u,v]$.
 \epr
  \subsubsection{Stability conditions \texorpdfstring{$\sigma$}{\space} on \texorpdfstring{$K(l)$}{\space} with
  \texorpdfstring{$P_\sigma=R_{{v},\Delta_{l+}}$}{\space}} \label{subsection for Kronecker}
In this subsection $l\geq 3$ is fixed and $Q=K(l)$, $\Delta_{l+}$ is the positive root system of $K(l)$, $\mc A=Rep_\CC(Q)$, $\mc T= D^b(Rep_\CC(Q))$.  \vspace{3mm}

 For a representation $X= \begin{diagram}[w=3em]
\CC^n & \upperarrow{}
\lift{-2}{  \vdots }
\lowerarrow{} & \CC^m
\end{diagram} \in {\mc A}$ we write $\ul{\dim}(X)=(n,m)$, $\ul{\dim}_1(X)=n$,\vspace{3mm}

 $\ul{\dim}_2(X)=m$. The simple objects of the standard t-structure ${\mc A}=Rep_\CC(Q)\subset D^b(Rep_\CC(Q))$ are: \vspace{2mm}
\begin{gather} \label{s_1 s_2 kron} s_1 = \begin{diagram}[w=3em]
\CC & \upperarrow{}
\lift{-2}{  \vdots }
\lowerarrow{} & 0
\end{diagram}, \ \ \   s_2 = \begin{diagram}[w=3em]
0 & \upperarrow{}
\lift{-2}{  \vdots }
\lowerarrow{} & \CC.
\end{diagram} \end{gather} \vspace{1mm}
To $\mc A \subset \mc T$ we can apply Remark \ref{HN property and locally finiteness} and then we have $\HH^{\mc A}\subset \st(\mc T)$ and
bijection $\HH^{\mc A} \ni (\mc P, Z) \mapsto (Z(s_1),Z(s_2)) \in \HH^2$.

 For any  $ (\mc P, Z) \in \HH^{\mc A} $ , $t\in (0,1]$     $\mc P(t)$ consists of the objects in ${\mc A}$  satisfying the condition in Definition
 \ref{Z-semistable}. If we denote  ${v}=(Z(s_1),Z(s_2)) \in \HH^2$, then  by $Z(X)=({v},\ul{\dim}(X))$:
\begin{gather} X \in {\mc P}(t), t \in (0,1]  \iff  \nonumber \\[-2mm] \label{sigma ss kron} \\[-2mm] \mbox{for any}  \  {\mc A}\mbox{-monic } \ X' \rightarrow X \
 \  \quad \arg({v},\ul{\dim}(X'))\leq  \arg({v},\ul{\dim}(X))= \pi t. \nonumber \end{gather}

 \begin{lemma} \label{density for kronecker0}  Let $ \sigma=(\mc P, Z) \in \HH^{\mc A} $ and $\arg(Z(s_1))>\arg(Z(s_2))$.
  Let $(n,m)\in \Delta_{l+}$ be a Schur root.  Then $\frac{n z_1 +m z_2}{\abs{n z_1 +m z_2}} \in P_{ \sigma }$, where $z_i=Z(s_i)$, $i=1,2$.
 \end{lemma}
 \bpr So, let $(n,m)\in \Delta_{l+}$ be a Schur root. We  show that  there exists a  $\sigma$-semistable $X$ with $\ul{\dim}(X)=(n,m)$. Then the lemma
 follows  because $X\in \mc P(t)\neq \{0\}$ for some $t\in (0,1]$ and by the formula $ m(X) \exp(\ri \pi t)=Z(X)=n z_1 + m z_2 $.

 If $m=0$, then $n=1$ (recall \eqref{roots for K(l)})  and then $X=s_1$ is the semistable, which we need (it is even stable in $\sigma$, since it is a simple object
  in $\mc A$). Hence we can assume that $m\geq 1$.

Denote $\arg(z_i)=\phi_i$, $i=1,2$,  ${v}=(z_1,z_2)$. Then  $0 < \phi_2 < \phi_1 \leq \pi$. By \eqref{arg(alpha...)}
for any $X$ with $\ul{\dim}(X)=(n,m)$ we have  $\arg({v},\ul{\dim}(X) )=\arg(n z_1 + m z_2)=f(n/m)$. Then by \eqref{sigma ss kron} such a $X$
is semi-stable in $\sigma$ iff  any ${\mc A}$-monic $X'\rightarrow X$   satisfies
$$ \arg( \ul{\dim}_1(X') z_1 + \ul{\dim}_2(X') z_2) \leq f\left (\frac{n}{m}\right). $$ Recall that $f(n/m)<\phi_1$ (see Remark \eqref{remark for z_1,z_2}).
From the last inequality  we get $\ul{\dim}_2(X')\neq 0$ and then by \eqref{arg(alpha...)} this inequality  can be rewritten as
$f(\ul{\dim}_1(X')/\ul{\dim}_2(X')) \leq  f\left (n/m\right)$.

 So, we see that $X \in {\mc A}$ with $\ul{\dim}(X)=(n,m)$ is  $\sigma$-semistable   iff
  any $\mc A$-monic arrow $X'\rightarrow X$ satisfies: \begin{gather} \label{semistable by n,m}  \ul{\dim}_2(X')\neq 0, \quad \frac{\ul{\dim}_1(X')}{\ul{\dim}_2(X')} \leq  \frac{n}{m}. \end{gather}
Now since $(n,m)$  is a Schur root then by Proposition \ref{prop king} there exists $X\in {\mc A}$ with $\ul{\dim}(X)=(n,m)$ and a non-zero
$\theta:K_0({\mc A}) \rightarrow \RR $, s. t. $X$ is $\theta$-semistable (see definition \ref{theta stability}).   We will
show that this $X$  is the $\sigma$-semistable, which we need.

By $\theta$-semistability of $X$  we have $\theta(1,0) n +
\theta(0,1) m = 0$.  By $m\neq 0$ we have a monic
map\footnote{Since the  vertex corresponding  to $s_2$ is a sink,
then $s_2$ is  a subobject of any $X\in Rep_\CC(K(l))$ with
$\ul{\dim}_2(X)\neq 0$.} $s_2 \rightarrow X$ and then again by
$\theta$-semistability  $\theta(0,1)\geq 0$, which together with
$\theta(1,0) n + \theta(0,1) m = 0$, $\theta \neq 0$ implies
   \ben \theta(1,0)<0, \ \  \theta(0,1)>0, \ \ \theta(1,0) \frac{n}{m} + \theta(0,1)  = 0.  \een Let us take now any monic arrow
   $X'\rightarrow X$ in ${\mc A}$ with $X'\neq 0$. By $\theta$-semistability $0\leq \theta(X')=\theta(1,0) \ul{\dim}_1(X')+ \theta(0,1) \ul{\dim}_2(X')$.
   Hence by $\theta(1,0)<0$ we obtain   $\ul{\dim}_2(X')\neq 0$. Therefore we can write
   \ben \theta(1,0)\frac{ \ul{\dim}_1(X')}{\ul{\dim}_2(X')}+ \theta(0,1) \geq 0 = \theta(1,0) \frac{n}{m} + \theta(0,1).  \een
By $\theta(1,0)<0$ it follows    $\frac{ \ul{\dim}_1(X')}{\ul{\dim}_2(X')}\leq \frac{n}{m}$. Whence, we verified \eqref{semistable by n,m} and the lemma follows.
  \epr
 \begin{coro} \label{density for kronecker} Let $ \sigma=(\mc P, Z) \in \HH^{Rep_\CC(K(l))} \subset \st D^b(Rep_\CC(K(l)))$ and $\arg(Z(s_1))>\arg(Z(s_2))$.
  Then $ P_{ \sigma }=R_{{v},\Delta_{l+}} $, where ${v}=(Z(s_1),Z(s_2))$.
 \end{coro}
 \bpr The good luck is that  all elements of $\Delta_{l+}$ are Schur roots. Indeed \cite[Theorem 4 a)]{Kac} says that $\Delta_{l+}^{im}$ are Schur roots.
 In \cite{Ringel} one can read that the indecomposable representations of $Q$ with dimension vectors real roots are  also Schur.
 Whence, we see that any $(n,m) \in \Delta_{l+}$ is a Schur root and we can apply the previous lemma to it. The corollary follows.
 \epr

\begin{remark} Recently it was noted in   \cite{GLMMN} that there is a connection between    \cite{Weist1}, \cite{Weist2}, \cite{Reineke}  and the  density in an arc for the  Kronecker quiver.

\end{remark}

\subsection{Kronecker pairs} \label{kronecker pairs}
In this subsection we generalize  Corollary  \ref{density for kronecker}.
The most general statement is Theorem \ref{coro for kronecker pairs 0}, but we use   further only its  Corollary   \ref{coro for kronecker pairs 3}
(corollary \ref{coro for kronecker pairs 2} is  intermediate). The first step  is: \footnote{It is motivated by the Bondal's result in  \cite{Bondal} for equivalence between  triangulated category generated by a strong exceptional
  collection and the derived category of modules over an algebra of homomorphisms of this collection  and by a note on this equivalence
  in \cite{Macri}. Observe however that we do not have restriction on  $(E_1,E_2)$ to be a strong pair  and we construct equivlanece between $t$-structures.}

\begin{lemma} \label{from pair to kronecker quiver} Let $\mc T$ be a $k$-linear triangulated category, where $k$ is any field. Let $(E_1,E_2)$ be a full exceptional
 pair, s. t. $\Hom^{\leq 0}(E_1,E_2)=0$,  $0<\dim_k(\Hom^{1}(E_1,E_2))=l<\infty$. Let $\mc A$ be the extension closure of $(E_1,E_2)$  in $\mc T$.

Then $\mc A$ is a heart of a bounded t-structure in $\mc T$ and there exists an   equivalence of abelian categories:
$F: \mc A \rightarrow Rep_k(K(l))$, s. t. $F(E_1)=s_1$, $F(E_2)=s_2$ \ ($s_1$, $s_2$ are as in \eqref{s_1 s_2 kron}).
\end{lemma}
\bpr In \cite[p. 6]{collinsP} or \cite[section 3]{rickard} it is shown that by  $\hom^{\leq 0}(E_1,E_2)=0$ and\footnote{If $S$ is a
subset of objects in a triangulated category $\mc T$ we denote by
$\left\langle S \right \rangle $ the triangulated subcategory
generated by $S$.}   $\mc T = \left\langle E_1 ,  E_2 \right\rangle$ it follows that  $\mc A$ is a heart of a bounded t-structure of $\mc T$ (see also \cite[section 8]{KellerAndPedro}). In particular $\mc A$ is an abelian category.

 Let $DT(\mc T)$ be the category of distinguished triangles
in $\mc T$ (objects are  the distinguished triangles and morphisms
are  triple of arrows  between triangles making commutative the
corresponding diagram).  Using  the semi orthogonal decomposition   $\mc T = \left\langle \left\langle E_1
\right\rangle, \left\langle E_2 \right\rangle \right\rangle$ one
can construct three functors: \be G: \mc T \rightarrow DT(\mc T),
\ \  \lambda_1 : \mc T \rightarrow \left \langle  E_1 \right
\rangle, \ \  \lambda_2 : \mc T \rightarrow \left \langle  E_2
\right \rangle \ee s. t.
  the triangle $G(X) \in DT(\mc T)$ for any $X \in \mc T$ is:
  \begin{gather} \label{G(X)}G(X) \  =  \  \begin{diagram} \lambda_2(X) & \rTo^{u_X} & X  & \rTo^{v_X} &  \lambda_1(X) & \rTo^{w_X}  & \lambda_2(X)[1] \end{diagram}
   , \ \  \lambda_1(X) \in \left\langle E_1 \right\rangle ,  \lambda_2(X)  \in \left\langle E_2 \right\rangle. \end{gather}   It is well known  that
     $\lambda_2$  is right adjoint to the embedding functor $\left \langle E_{2} \right \rangle \rightarrow \mc T$ and $\lambda_1$ is left adjoint to
      $\left \langle E_{1} \right \rangle \rightarrow \mc T$ (see for example \cite[p. 279]{GM}). The adjoint functor (left or right) to an  exact functor
      is also an exact functor (\cite[Proposition 1.4]{BK}). Therefore $\lambda_1$ and $\lambda_2$ are exact functors.
   If we restrict $\lambda_1$, $\lambda_2$ to $\mc A$ then we
  obtain exact functors between abelian categories
  \be \nonumber \lambda_i^{\mc A}: \mc A \rightarrow \mc A_i  \qquad i=1,2, \ee where $\mc A_i\cong k\mbox{-}Vect$ is  the additive closure of $E_i$.

We define the functor $F: \mc A \rightarrow Rep_k(K(l))$ as follows. First choose a basis of $\Hom^1(E_1,E_2)$ and a decomposition of any
$Y \in \mc A_i$ into $\dim(\Hom(E_i,Y))$ number of copies of $E_i$, $i=1,2$.
Take any $X\in \mc A$, then we get a  distinguished triangle $G(X)$ as in \eqref{G(X)} with $\lambda_i(X)=\lambda_i^{\mc A}(X)$,
in particular we get an arrow  $\bd \lambda_1^{\mc A}(X) & \rTo^{w_X}  & \lambda_2^{\mc A}(X)[1] \ed $. This arrow, using the chosen decompositions
and the basis of $\Hom^1(E_1,E_2)$, can be expressed  by  $l$   $a_2\times a_1$ matrices over $k$, where  $a_i = \dim(E_i,\lambda_i(X))$, $i=1,2$.
In particular these $l$ matrices are a representation of $K(l)$ with dimension vector $(a_1,a_2)$ and we define   $F(X)$ to be this representation.

 Let $f:X \rightarrow Y $ be an arrow in $\mc A$ then, as far as  $G: \mc T \rightarrow DT(\mc T)$ is a functor,  $G(f)$ is a morphism of triangles,
 hence the diagram:  $ \bd[size=1.5em] \lambda_1^{\mc A}(X) & \rTo^{w_X}  & \lambda_2^{\mc A}(X)[1] \\
                     \dTo^{\lambda_1^{\mc A}(f)} &   &   \dTo^{\lambda_2^{\mc A}(f)[1]}  \\
\lambda_1^{\mc A}(Y) & \rTo^{w_Y}  & \lambda_2^{\mc A}(Y)[1]  \ed $    is commutative.
Let $M_1$, $M_2$ be the matrices of  $\lambda_1^{\mc A}(f)$, $\lambda_2^{\mc A}(f)$. The commutativity of the diagram above  implies
that $(M_1,M_2):F(X) \rightarrow F(Y)$ is an arrow in $Rep_k(K(l))$ and  our definition of $F(f)$ is  $F(f)=(M_1,M_2)$.

  By the exactness of $\lambda_i^{\mc A}$, $i=1,2$ it follows that  $F$ is an exact functor between abelian categories. Now, by straightforward  computations one can show that $F$ is  an equivalence.\epr
This lemma prompts the following definition
\begin{df} \label{defKroneckeppair} A pair of objects $(E_1,E_2)$ in a  $k$-linear ($k$ is any field)  triangulated category   $\mc T$ is called \ul{Kronecker pair} if:
\begin{itemize}
    \item  $(E_1,E_2)$ is an exceptional pair
    \item  $\Hom^{\leq 0}(E_1,E_2)=0$
    \item  $3 \leq \dim_k(\Hom^{1}(E_1,E_2)) < \infty $.
\end{itemize}
\end{df}

\begin{coro} \label{coro for kronecker pairs 1} Let $(E_1,E_2)$ be  a  Kronecker pair in a  $\CC$-linear triangulated category   $\mc D$.
Denote $l=\dim(\Hom^1(E_1,E_2))$, $\mc T=\left \langle E_1, E_2 \right \rangle \subset \mc D$ and $\mc A$ - the extension closure of $(E_1,E_2)$.

 Then any $ \sigma=(\mc P, Z) \in \HH^{\mc A} \subset \st D^b(\mc T)$ with $\arg(Z(E_1))>\arg(Z(E_2))$ satisfies $P_{ \sigma }=R_{{v},\Delta_{l+}} $,
 where ${v}=(Z(E_1),Z(E_2))$. In particular  $ P_{ \sigma }$ is dense in an arc of non-zero length.
\end{coro}
\bpr   We take the  equivalence $F : \mc A \rightarrow Rep_\CC(K(l))$ constructed in Lemma \ref{from pair to kronecker quiver},
$\mc A \subset \mc T$, $Rep_\CC(K(l)) \subset D^b(Rep_\CC(K(l)))$. This equivalence induces  a natural   bijection
$F^*:\HH^{Rep_\CC(K(l))} \rightarrow \HH^{\mc A}$.
For  $\sigma=(\mc P, Z) \in  \HH^{\mc A}$, $\sigma' = (\mc P', Z')\in \HH^{Rep_\CC(K(l))}$, from $F^*(\sigma')=\sigma$ it follows
$Z(E_i)=Z'(s_i)$ (because $F(E_i)=s_i$) and  $ P_\sigma =P_{\sigma'} $.   Then the corollary follows from Corollary \ref{density for kronecker}.
\epr
Whence in this Corollary we obtained $\sigma \in \st(\left \langle E_1, E_2 \right \rangle)$ with $P_\sigma$ dense in a nontrivial arc.
 To obtain $\sigma' \in \st(\mc D)$ with such a property, we want to extend the given  $\sigma \in \st(\left \langle E_1, E_2 \right \rangle)$
  to a stability condition on $\mc D \supset \mc T$ in the following sense:
\begin{df} Let $\mc T \subset \mc D$ be a triangulated subcategory in a triangulated category $\mc D$.
 We say that  $\sigma=(\mc P,Z) \in \st(\mc T)$  can be extended to $\mc D$ (or extendable to $\mc D$) if there exists
 $\sigma_e=(\mc P_e,Z_e) \in \st(\mc D)$, s. t.  $\ \ \ {Z_e}\circ K_0(\mc T \subset \mc D)=Z$ and  $\{{\mc P}(t) \subset \mc P_e(t)\}_{t\in\RR}$.
  In this case $\sigma_e$ is called extension of $\sigma$.
 \end{df}

\begin{remark} \label{from extension to phases} From  Definition \ref{P_sigma} it follows that if  $\sigma_e$ is an extension of $\sigma$, then
$ P_{\sigma_e} \supset P_\sigma $.
\end{remark}
By Corollary \ref{coro for kronecker pairs 1} it follows:
 \begin{theorem} \label{coro for kronecker pairs 0} Let $(E_1,E_2)$ be  a  Kronecker pair in a  $\CC$-linear triangulated category   $\mc D$.
 Denote $l=\dim(\Hom^1(E_1,E_2))$, $\mc T=\left \langle E_1, E_2 \right \rangle \subset \mc D$ and $\mc A$ - the extension closure of $(E_1,E_2)$.

 Then any $ \sigma \in \st D^b(\mc D)$, which is an extension  of a stability condition  $(\mc P, Z) \in \HH^{\mc A} \subset \st D^b(\mc T)$ with
 $\arg(Z(E_1))>\arg(Z(E_2))$ satisfies $P_{ \sigma }\supset R_{{v},\Delta_{l+}} $, where ${v}=(Z(E_1),Z(E_2))$. In particular  $ P_{ \sigma }$ is dense
 in an arc of non-zero length.
 \end{theorem}

One setting, where we can extend these stability conditions, is as
follows.

Assume that $(E_0,E_1,\dots,E_n)$ is a full Ext-exceptional collection\footnote{The ``Ext-'' means $\Hom^{\leq 0}(E_i,E_j)=0$ for $0\leq i<j\leq n$.} in $\mc D$.
 Then for any $0\leq i<j\leq n$ the extension closure $\mc A_{ij}$ of $E_i,E_{i+1},\dots,E_j$ is a heart of a bounded $t$-structure in
 $\mc T_{ij}=\left \langle E_i,E_{i+1},\dots,E_j \right \rangle \subset \mc D$ (see \cite[Lemma 3.14]{Macri},\cite{collinsP}), hence we have a corresponding
 family $\HH^{{\mc A}_{ij}} \subset \st(\mc T_{ij})$. In this setting all  stability conditions in $\HH^{{\mc A}_{ij}} $ are extendable to $\mc D$. The precise
 statement is (see  \cite{DK}  and  \cite[Proposition 3.17]{Macri}) that    there is a  surjective map\footnote{We denote here ${\mc A} ={\mc A}_{0 n}$.} $\pi_{ij}:\HH^{\mc A} \rightarrow \HH^{\mc A_{ij}}$,
  s. t.  for any    $\sigma \in \HH^{\mc A_{ij}}$, $\sigma_e \in \HH^{\mc A}$ from $\pi_{ij}(\sigma_e)=\sigma$  it follows that $\sigma_e$ is an extension
  of $\sigma$.  Having the desired extensions we obtain by Corollary \ref{coro for kronecker pairs 0}:
\begin{coro} \label{coro for kronecker pairs 2} Let $(E_0,E_1,\dots,E_n)$ be a full Ext-exceptional collection in a $\CC$-linear triangulated category
$\mc D$. Let     $(E_i,E_{i+1})$ be a Kronecker pair for some $0\leq i \leq n-1$.  Denote $l=\dim(\Hom^1(E_i,E_{i+1}))$ and the extension closure of
$(E_0,E_1,\dots,E_n)$ by  $\mc A$.

 Then any $ \sigma=(\mc P, Z) \in \HH^{\mc A}$ with $\arg(Z(E_i))>\arg(Z(E_{i+1}))$ satisfies $P_{ \sigma }\supset R_{{v},\Delta_{l+}} $, where
 ${v}=(Z(E_i),Z(E_{i+1}))$. In particular  $ P_{ \sigma }$ is dense in an arc of non-zero length.
\end{coro}

\begin{coro} \label{coro for kronecker pairs 3} Let $(E_0,E_1,\dots,E_n)$ be any full exceptional collection in a $\CC$-linear triangulated category
$\mc D$ of finite type\footnote{by finite type we mean that for any pair $X,Y \in \mc T$ we have $\sum_{k\in \ZZ}\dim(\Hom(X,Y[k])) < \infty$.}. Let
  $(E_i,E_{j})$ be a Kronecker pair for some $0\leq i <j \leq n$.  Then there exists  a  family of stability conditions $\sigma$ on $\mc T$ for which
   $ P_{ \sigma }$ is dense in an arc of non-zero length.
\end{coro}\bpr
First by mutations of the exceptional collection $(E_0,E_1,\dots,E_n)$ we can obtain a full exceptional collection $(E_i,E_j,C_2,\dots,C_n)$. Then,
because $\mc T$ is of finite type, after shifts of $C_2,C_3,\dots,C_n$ we can  obtain a full exceptional collection $\mc B = \{B_0,B_2,\dots, B_n \}$,
which is Ext and $B_0=E_i$, $B_1=E_j$. So we get a full Ext-exceptional collection  $\mc B$ for which $(B_0,B_1)$ is a Kronecker pair. Now if we denote
by $\mc A$ the extension closure of  $\mc B$ by Corollary \ref{coro for kronecker pairs 2} it follows that any $\sigma=(\mc P,Z) \in \HH^{\mc A}$  with
$\arg(Z(B_0))>\arg(Z(B_1))$ satisfies $P_\sigma \supset R_{v,\Delta_{l+}}$, where $v=(Z(B_0),Z(B_1))$ and $l=\dim_\CC(\Hom^1(B_0,B_1))$.
\epr
\begin{remark} \label{more general} More general setting, where the stability conditions $\HH^{\mc A}$ in Theorem \ref{coro for kronecker pairs 0} can be extended, is that there exists  a semi-orthogonal decomposition $( \mc D',  \left \langle E_1,E_2 \right \rangle )$ of $\mc D $ with additional assumptions, specified in \cite[Theorem 3.6]{collinsP} and  \cite[Proposition 3.5]{collinsP}.
\end{remark}
\subsection{Application to quivers}

In this subsection we apply the results of the previous subsection \ref{kronecker pairs}  to quivers and obtain Corollary
\ref{upper bound for euc}, Proposition \ref{non euc and non dyn}. Table  \eqref{table} contains Proposition \ref{non euc and non dyn} and the results of subsection
\ref{Euc and Kron}.

Let $Q$ be an acyclic quiver. The notations $V(Q)$, $Arr(Q)$, $\Gamma(Q)$ are explained in subsubsection \ref{quivers and Kac}.
It   is shown in \cite{WCB1}  that any exceptional collection $(E_1,E_2,\dots,E_n)$ in $Rep_\CC(Q)$ of length $n=\#(V(Q))$  is a  full exceptional collection
of $D^b(Rep_\CC(Q))$. Furthermore, any exceptional collection $(E_1,E_2,\dots,E_i)$ in $Rep_\CC(Q)$  with $i<n$ can be completed to a full
$(E_1,E_2,\dots,E_i, E_{i+1},\dots, E_n)$  exceptional collection. In particular if we are given a Kronecker pair in $Rep_\CC(Q)$ we can complete it to a
 full exceptional collection and,  since $D^b(Rep_\CC(Q))$ is of finite type,  we  can apply  Corollary \ref{coro for kronecker pairs 3}.
  So that only existence of a Kronecker pair in $Rep_\CC(Q)$ is enough to apply Corollary \ref{coro for kronecker pairs 3} and to construct $\sigma$ with
  $P_\sigma$ dense in an arc.  Now  using  Corollaries \ref{Dynkin},  \ref{coro for density1},  we can  easily prove:
  \begin{coro} \label{upper bound for euc} Let   $Q$ be either an Euclidean  or a Dynkin  quiver.  Then any exceptional pair $(E_1,E_2)$ in
  $Rep_\CC(Q)$ satisfies  $\dim_\CC(\Hom(E_1,E_2)) < 3$,$\dim_\CC({\rm Ext}^1(E_1,E_2)) < 3$.
  \end{coro}
  \bpr Since $Rep_\CC(Q)$ is hereditary then  the exceptional objects in $D^b(Rep_\CC(Q))$ are just shifts of exceptional objects in
  $Rep_\CC(Q)$ and then from the arguments above  and Corollaries \ref{Dynkin}, \ref{coro for density1} it follows that there does not exists a Kronecker pair in $D^b(Rep_\CC(Q))$.
   In other words for any exceptional pair $(E_1,E_2)$  in $D^b(Rep_\CC(Q))$  the minimal nonzero degree $\Hom^{min}(E_1,E_2)\neq 0$, $\Hom^{< min}(E_1,E_2)= 0$
    has dimension $\dim_\CC(\Hom^{min}(E_1,E_2))\leq 2$. Since $Rep_\CC(Q)$ is hereditary then there are at most two nonzero degrees in $\Hom^*(E_1,E_2)$ and it
     remains to show that the maximal nonzero degree $\Hom^{max}(E_1,E_2)\neq 0$, $\Hom^{> max}(E_1,E_2)= 0$ has dimension $\dim_\CC(\Hom^{max}(E_1,E_2))\leq 2$.

   For any exceptional pair $(E_1,E_2)$  it is well known that $(L_{E_1}(E_2), E_1)$ is also an exceptional pair, where $L_{E_1}(E_2)$ is determined by the
    distinguished triangle:
  \be    \begin{diagram} L_{E_1}(E_2) & \rTo & \Hom^*(E_1,E_2)\otimes E_1  & \rTo^{ev_{E_1, E_2}} &  E_2 & \rTo  & L_{E_1}(E_2)[1]. \end{diagram} \ee
Take any $i \in \ZZ$. We  show below that
$\Hom^{i}(L_{E_1}(E_2),E_1) \cong \Hom^{-i}(E_1,E_2)$, which means
that the maximal  non-zero degree of $\Hom^*(E_1,E_2)$   is the
minimal non-zero degree of\\ $\Hom^{*}(L_{E_1}(E_2),E_1)$  and
they are isomorphic. Then the corollary follows by the proved
inequality for the minimal non-vanishing degrees.

 We apply $\Hom^i(\_,E_1)$ to the triangle above and by $\Hom^*(E_2,E_1)=0$ it follows
 \be \label{proof of upper bound eq} \Hom^i(\Hom^*(E_1,E_2)\otimes E_1, E_1) \cong \Hom^i(L_{E_1}(E_2), E_1). \ee
  On the other hand (recall that $E_1$ is an exceptional object)
  \begin{gather} \nonumber \Hom^*(E_1,E_2)\otimes E_1 \cong \bigoplus_{j} E_1[-j]^{\dim(\Hom^j(E_1,E_2))}\ \ \ \ \  \Rightarrow \\
 \nonumber \Hom^i(\Hom^*(E_1,E_2)\otimes E_1, E_1)\cong \Hom(\oplus_{j} E_1[-j]^{\dim(\Hom^j(E_1,E_2))}, E_1[i])\cong \CC^{\dim(\Hom^{-i}(E_1,E_2))}, \end{gather}
which together with \eqref{proof of upper bound eq} give $\Hom^{i}(L_{E_1}(E_2),E_1) \cong \Hom^{-i}(E_1,E_2)$ and the corollary is proved.
   \epr

Next we want to prove:
 \begin{prop} \label{non euc and non dyn}
  Any acyclic connected  quiver $Q$, which is neither Euclidean nor Dynkin has a family of stability conditions $\sigma$ on $D^b(Rep_\CC(Q))$, s. t.
   $ P_{ \sigma }$ is dense in an arc of non-zero length.
   \end{prop}
\begin{remark} \label{oriented cycles} If there are oriented cycles in $Q$, then one can show that there is  a  family\footnote{For example the representations $\{ \bd[1em] \CC &\rTo^{\lambda}&\CC \ed \}_{\lambda \in \CC}$ of the quiver with one vertex and one loop  are all simple and mutually non isomorphic} $\{s_\lambda \}_{\lambda \in \CC}$ of non isomorphic simple objects in ${\mc A}=Rep_\CC(Q)$. Then if  we define for simple object $s \in \mc A$ \be Z(s)=\left \{ \begin{array}{c c} \frac{\lambda}{\abs{\lambda}} & \mbox{if} \ \ s=s_\lambda, \lambda \in \HH \\
                                         \ri & \mbox{otherwise}\end{array} \right. \ee
 we obtain a stability function $Z:K_0({\mc A}) \rightarrow \CC$, which has HN property, since $\mc A$ is of finite length. One can show that the corresponding
 stability condition $\sigma=({\mc P}, Z) \in \HH^{\mc A}$ is locally finite.   Since all $\{ s_\lambda \}_{\lambda \in \CC}$ are simple in $\mc A$, then they are
 $\sigma$-semistable. Hence $P_\sigma = S^1$.
\end{remark}
So let us fix a quiver $Q$, satisfying the conditions of  Proposition \ref{non euc and non dyn}.    By the arguments given in the beginning of this subsection
(before Corollary \ref{upper bound for euc}) we reduce the proof to finding a Kronecker pair in $D^b(Rep_\CC(Q))$.  From here till the end of this subsubsection we
present the proof of the following:
\begin{prop} \label{prop for kronecer pairs in quivers} Any $Q$, satisfying the conditions in Proposition \ref{non euc and non dyn}, has   a Kronecker pair
in $Rep_\CC(Q)$ (i. e.  a  pair or representations  $(\rho,\rho')$ in $Rep_\CC(Q)$ with $\Hom_{\mc T}^*(\rho',\rho)=\Hom_{\mc T}^{\leq 0}(\rho, \rho')=0$, $\dim_\CC(\Hom_{\mc T}^1(\rho,\rho'))\geq 3$, where ${\mc T}= D^b(Rep_\CC(Q))$).
\end{prop}

   Recall (see page 8 in \cite{WCB2}) that for $\rho,\rho'\in Rep_\CC(Q)$,  we have the formula\footnote{throughout the proof  $\Hom^i(,)$  means $\Hom^i_{\mc T}(,)$, ${\mc T}= D^b(Rep_\CC(Q))$}
\begin{gather}\label{euler} \dim_\CC(\Hom(\rho,\rho'))-\dim_\CC(\Hom^1(\rho,\rho'))=\scal{\ul{\dim}(\rho),\ul{\dim}(\rho')}_Q,  \end{gather}
where $\scal{,}_Q$ is defined in \eqref{origin,end}, \eqref{euler form}. Let us denote  for $\rho \in Rep_\CC(Q)$:
\be  supp(\rho) = supp(\ul{\dim}(\rho)) = \{i\in V(Q)\vert \ul{\dim}_i(\rho)\neq 0 \}. \ee
For $i\in V(Q)$   the simple representation $s_i$ is characterized by $supp(s_i)=\{i\}$, $\ul{\dim}_i(s_i)=1$. Obviously $\{ s_i\vert i \in V(Q) \}$ are exceptional objects in $D^b(Rep_\CC(Q))$. One of the representations in the Kronecker pair $(\rho, \rho')$, which we shall obtain, is among the exceptional objects $\{ s_i\vert i \in V(Q) \}$.

It is useful to denote
\begin{gather} A,B \subset V(Q)  \  \ Arr(A,B) = \{a \in Arr(Q) \vert s(a)\in A, t(a) \in B\}, \nonumber\\[-2mm] \label{Arr(A,B)} \\ Ed(A,B)=Ed(B,A)=Arr(A,B)\cup Arr(B,A). \nonumber\end{gather}

To find Kronecker pairs in $D^b(Rep_\CC(Q))$ we observe first,  that for $\rho, \rho' \in Rep_\CC(Q)$ we have $\Hom^{\leq -1}(\rho,\rho')=0$
in $D^b(Rep_\CC(Q))$ and
\begin{gather}  supp(\rho)\cap supp(\rho')=\emptyset \  \ \ \Rightarrow \ \ \ \Hom(\rho,\rho')=\Hom(\rho',\rho)=0, \nonumber \\[-2mm] \label{suppint} \\
\dim_\CC(\Hom^1(\rho,\rho'))= \sum_{ a \in Arr(supp(\rho), supp(\rho'))} \ul{\dim}_{s(a)}(\rho) \ul{\dim}_{t(a)}(\rho') \nonumber \end{gather}  which follows by \eqref{euler}. Another useful statement is
\begin{lemma} \label{find exc with big dim} Let $Q$ be  an Euclidean quiver.
Then for each  $n\in \NN$  there exists an exceptional representation $\rho \in Rep_\CC(Q)$, s. t. $\ul{\dim}_v(\rho)\geq n$ for each $v \in V(Q)$.
\end{lemma}
\bpr Let  $\delta \in \NN_{\geq 1}^{V(Q)}$     be the minimal imaginary root of $\Delta_+(Q)$, used in the proof of Lemma \ref{finite number sonv seq}\footnote{In  \cite[fig. (4.13)]{Barot} are given the coordinates of $\delta$ for all euclidean graphs}.  One property of $\delta$ is  that $(\_,\delta )_Q = 0$  on $\NN^{V(Q)}$, where $(\alpha,\beta)_Q= \scal{\alpha,\beta}_Q +\scal{\beta,\alpha}_Q $ is the symmetrization of $\scal{,}_Q$.  We find below a vertex $v\in V(Q)$, s. t. $\scal{1_v,1_v}_Q=\frac{1}{2}(1_v,1_v)_Q=1$, $\scal{1_v,\delta}_Q \neq 0$.
 Then for any $m\in \NN$   we have $1=\frac{1}{2}(1_v+m \delta,1_v+m \delta)_Q=\scal{1_v+m \delta,1_v+m \delta}_Q= 1$, $\scal{1_v+m \delta,\delta}_Q\neq 0$, hence, for big enough $m$, $r=1_v+m \delta$ is  a real positive root $r \in \Delta_+(Q)$ with  $\scal{r,\delta}_Q\neq 0$,  $\{ r_v \geq n \}_{v \in V(Q)}$. Hence by  \cite[p. 27]{WCB2} there is an exceptional representation $\rho$ with $\ul{\dim}(\rho)=r$ and the lemma follows.

If\footnote{Recall that by $\Gamma(Q)$ we denote the underlying non-oriented graph and that $\wt{A}_1$ is graph with two vertices and two parallel edges connecting them,  $\wt{A}_m$ with $m\geq 2$ is a loop with $m+1$ vertices and $m$ edges connecting them forming a simple loop}  $\Gamma(Q)=\wt{A}_m$, $m\geq 1$:  As far as  $Q$ is not an oriented cycle then there is a sink $s \in V(Q)$ (i. e. both the arrows touching $s$ end at it). Hence by \eqref{euler form} $\scal{1_s,1_s}_Q =-\scal{1_s,\delta}_Q = 1$.

 If $\Gamma(Q)=\wt{D}_m$, $m\geq 4$, $\wt{E}_6, \wt{E}_7, \wt{E}_8$. In  \cite[fig. (4.13)]{Barot} are given the coordinates of $\delta$ for all these options for
 $\Gamma(Q)$.   We take $v\in V(Q)$ to be the extending vertex, in \cite[fig. (4.13)]{Barot} this vertex is denoted by $\star$, i.e. $v=\star$. Then by
  \eqref{euler form} and the given in \cite[fig. (4.13)]{Barot} coordinates of $\delta$ one computes     $\scal{1_v,1_v}_Q=1$, $\scal{1_v,\delta}_Q=\pm 1$,
   depending on whether $\star$ is a sink/source in $Q$.  \epr

\begin{df} \label{subquiver} Let $A\subset V(Q)$, $A\neq \emptyset$. By $Q_A$ we denote the quiver with $V(Q_A)=A$ and $Arr(Q_A)= Arr(A,A)=\{a \in Arr(Q) \vert s(a)\in A, t(a) \in A\}$. For any $\rho \in Rep_\CC(Q_A)$ we denote by the same  letter $\rho$ the representation in $  Rep_\CC(Q)$, which on  $A$, $Arr(A,A)$ coincides with $\rho$ and is zero elsewhere.

We say that a vertex $v \in V(Q)$ is adjacent to $Q_A$ if $v \not \in A$ and  $Ed( A, \{v\}) \neq \emptyset$.
\end{df}

\begin{remark} \label{remark for extending} If $\rho \in Rep_\CC(Q_A)$ is an exceptional representation then the corresponding extended representation in
$Rep_\CC(Q)$ is also exceptional.
\end{remark}
\begin{remark} \label{main prop of adjacent vert}
If $v \in V(Q)$ is adjacent to $Q_A$ then    $Arr(Q_{A\cup\{v\}})=Arr(Q_{A}) \cup Ed( A, \{v\})$. \end{remark}

In the following two corollaries we consider   a configuration of a subset $A \subset V(Q)$ and  an adjacent to it vertex $v \in V(Q)$,  s. t.
the arrows connecting $v$ and $A$ are all directed  either  from $v$ to $A$ or  from $A$ to $v$, which means that $v$ is either a source or a sink in
$Q_{A\cup \{v\}}$.

In Corollary \ref{coro for kronecker pairs} we show that if $Q_A$ is an Euclidean quiver then we get a Kronecker pair $(E_1,E_2)$ in $Rep_\CC(Q)$ with
 $\dim_\CC(\Hom^1(E_1,E_2))$ as big as we want (without additional assumption on the quiver $Q$).

In Corollary \ref{coro for kronecker pairs1} we show that if  $\Gamma(Q_A)$ is  either $A_n$($n\geq 1$) or $D_n$($n\geq 4$)  then, under the additional
assumption that there are at least three edges between $v$ and $A$, we get a Kronecker pair $(E_1,E_2)$ in $Rep_\CC(Q)$ with $\dim_\CC(\Hom^1(E_1,E_2))$ equal
to this number of edges.

\begin{coro} \label{coro for kronecker pairs} Let $A\subset V(Q)$ be such that  $Q_A $ is Euclidean. Let $v \in V(Q)$ be a vertex, which is adjacent to
$Q_A$ and  either a sink or a source in $Q_{A\cup\{v\}}$. Then for any $n \geq 3$ there exists a Kronecker pair $(E_1,E_2)$ in $Rep_\CC(Q)$ with
$\dim(\Hom^1(E_1,E_2)) \geq n$.
\end{coro}
\bpr From Lemma \ref{find exc with big dim} and Remark \ref{remark for extending} we get an exceptional representation $\rho \in Rep_\CC(Q)$, s. t.
 $supp(\rho) = A$ and $ \{ \ul{\dim}_i(\rho) \geq n \}_{i \in A}$.

If $v$ is a sink in $Q_{A\cup\{v\}}$ then  $Arr(\{v\}, A)=\emptyset $,  and, since $v$ is adjacent to $Q_A$,  $Arr( A, \{v\})\neq \emptyset $.
 From  $\{v\} \cap A = \emptyset$ and  \eqref{suppint} we get $\Hom^*( s_v,\rho)=0$,  $\Hom^{\leq 0}( \rho, s_v)=0$, $\dim_\CC(\Hom^1( \rho, s_v)) \geq n$.
 So that $( \rho, s_v)$ is the  Kronecker pair we need.

If $v$ is a source in $Q_{A\cup\{v\}}$  then the same arguments show that $(  s_v, \rho)$ is such a  Kronecker pair.
\epr

\begin{coro} \label{coro for kronecker pairs1} Let $A\subset V(Q)$ be such that  $\Gamma(Q_A)$ is either $ A_n$($n\geq 1$) or $D_n$($n\geq 4$).
Let $v \in V(Q)$ be  adjacent to $Q_A$ and  either a sink or a source in $Q_{A\cup\{v\}}$. Let $ \# (Ed(\{v\}, A))=n  \geq 3$. Then there exists a Kronecker
pair $(E_1,E_2)$ in $Rep_\CC(Q)$ with $\dim(\Hom^1(E_1,E_2)) = n$.
\end{coro}
\bpr
Using that $\Gamma(Q_A)$ is either  $A_n$($n\geq 1$) or $D_n$($n\geq 4$) we see that the representation $\rho$, with
$A \ni i \mapsto \CC$, $Arr(A,A) \ni a \mapsto Id_\CC$ and zero elsewhere is an exceptional representation in $Rep_\CC(Q)$ with
$supp(\rho) = A$ and $ \{ \ul{\dim}_i(\rho) =1 \}_{i \in A}$.

If $v$ is a sink in $Q_{A\cup \{v\}}$ then  $Arr(\{v\}, A)=\emptyset $ and $\# (Arr( A, \{v\}))=\# (Ed(\{v\}, A) ) = n.$   From
$\{v\} \cap A = \emptyset$ and  \eqref{suppint} we get $\Hom^*( s_v,\rho)=0$,  $\Hom^{\leq 0}( \rho, s_v)=0$, $\dim_\CC(\Hom^1( \rho, s_v)) =n$. So
that $( \rho, s_v)$ is the  Kronecker pair we need.

If $v$ is a source in $Q_{A\cup \{v\}}$ then the same arguments show that $(  s_v, \rho)$ is such a  Kronecker pair.
\epr
An immediate consequence of this corollary is
\begin{coro} \label{three parallel arrows} If there are  $n\geq 3$ parallel arrows in $Arr(Q)$. Then there exists a Kronecker pair
$(E_1,E_2)$ in $Rep_\CC(Q)$ with $\dim(\Hom^1(E_1,E_2))=n$.
\end{coro}
\bpr  Let these arrows  start  at
 a vertex $i$ and end  at  a vertex $j$,  then  $\# Arr(\{i\},\{j\})=n\geq 3$, $Arr(\{j\},\{i\})=\emptyset$ and we apply the previous Corollary
 to $A=\{j\}$, $v=\{i\}$.
\epr Whence,  we can assume that there are not more than two
parallel arrows in $Arr(Q)$. We consider next  the case that  two parallel arrows
do occur.
\begin{remark} \label{special cases}
In the considerations, that follow,   we refer
most often to Corollary \ref{coro for kronecker pairs} (i. e. then
we get Kronecker pairs with arbitrary big
$\dim_\CC(\Hom^1(E_1,E_2))$), but there are  three   situations,
where we need Corollary \ref{coro for kronecker pairs1}\footnote{we already used it once in Corollary \ref{three parallel arrows}} with the
minimal admissible number of edges connecting $v$ and $A$, namely
$3$,  and then the produced Kronecker pair is with minimal
possible $\dim_\CC(\Hom^1(E_1,E_2))=3$.

The quiver $Q_{A \cup
\{v\}}$ observed in these three special situations\footnote{in Corollary \ref{three parallel arrows} the quiver $Q_{A \cup
\{v\}}$ is the Kronecker quiver, i. e. $Q_{A \cup
\{v\}}=K(n)$, $n\geq 3$}, in which we use
Corollary \ref{coro for kronecker pairs1}, is as follows (we
denote the set $A$ by $A=\{a_1,a_2,\dots,a_n\}$):

\begin{gather}  \label{special sit} S_1= \begin{diagram}[1.5em]
   &       &  a_2  &       &    \\
   & \ruTo &    & \rdTo  &       \\
v  & \pile{\rTo  \\ \rTo }  &    &       &  a_1
\end{diagram}  \ \ \  S_2= \begin{diagram}[1.5em]
a_1 &\rTo &a_2 &\lTo &a_3\\
&\luTo &\uTo &\ruTo \\
& &v
\end{diagram} \ \  \ 
 S_3= \begin{diagram}[1.5em]
& &a_4 & \\
&\ruTo &\uTo  &\luTo \\
a_1 &    &a_2 &    &a_3.\\
&\luTo &\uTo &\ruTo \\
& &v
\end{diagram}. \nonumber     \nonumber    \ \ \ \end{gather}
In $S_1$ $\Gamma(Q_A)=A_2$, in $S_2$  $\Gamma(Q_A)=A_3$, in $S_3$  $\Gamma(Q_A)=D_4$.
\end{remark}


\subsubsection{If there are \texorpdfstring{$2$}{\space} parallel arrows in \texorpdfstring{$Arr(Q)$}{\space}}
 Let these two arrows  start  at a vertex $i$ and end  at  a vertex $j$. So that throughout this subsubsection we have $\# (Arr(\{i\},\{j\})) = 2$,  $  Arr(\{j\},\{i\}) = \emptyset$. Then (recall Definition \ref{subquiver}):
 \begin{gather} Q_{\{i,j\}}=\begin{diagram}[1em] i  & \pile{\rTo  \\ \rTo }  & j \end{diagram}, \ \  \ \Gamma(Q_{\{i,j\}})=\wt{A}_1. \end{gather}        By our assumption that $Q$ is  connected  and not Euclidean there is a vertex $k \in V(Q)$ which is adjacent to $Q_{\{i,j\}}$. If either  $Arr(\{k\}, \{i,j\}) = \emptyset $ or $Arr( \{i,j\}, \{k\}) = \emptyset $ then by Corollary \ref{coro for kronecker pairs} we get a Kronecker pair. Hence we can assume  $ Arr( \{k\}, \{i,j\} ) \neq \emptyset $ and $ Arr( \{i,j\}, \{ k \} ) \neq \emptyset.$  From the condition that there are not oriented cycles we reduce to \begin{gather} Arr(\{k\},\{j\} ) \neq \emptyset  \ \ Arr( \{i\}, \{k\} )\neq \emptyset \ \ \Rightarrow \ \ Arr(\{j\}, \{k\} )=  Arr(  \{k\},\{i\} ) = \emptyset. \end{gather}
 If $Arr(\{k\},\{j\} )$ has two elements then  $Q_{\{k,j\}}=\begin{diagram}[1em] k  & \pile{\rTo  \\ \rTo }  & j \end{diagram}$,   $\Gamma(Q_{\{k,j\}})=\wt{A}_1$ , $i$ is adjacent to $Q_{\{k,j\}}$ and $Arr(\{k,j\},\{i\})=Arr(\{k\},\{i\})\cup Arr(\{j\},\{i\})=\emptyset $, hence Corollary \ref{coro for kronecker pairs} produces a Kronecker pair. So that we can assume that $Arr(\{k\},\{j\} )$ has only one element.

 If $Arr(\{i\},\{k\})$ has two elements, then  $Q_{\{i,k\}}=\begin{diagram}[1em] i  & \pile{\rTo  \\ \rTo }  & k \end{diagram}$, $\Gamma(Q_{\{i,k\}})=\wt{A}_1$, $j$ is adjacent to $Q_{\{i,k\}}$ and $Arr(\{j\},\{i,k\})=Arr(\{j\},\{i\})\cup Arr(\{j\},\{k\})=\emptyset $, hence Corollary \ref{coro for kronecker pairs} produces a Kronecker pair. Hence we can assume that $Arr(\{i\},\{k\})$, $Arr(\{k\},\{j\} )$ are single element sets. Using that $Q_{\{i,j\}}=\begin{diagram}[1em] i  & \pile{\rTo  \\ \rTo }  & j \end{diagram}$  we obtain that $Q_{\{i,j,k\}}$ is the same as  the quiver $S_1$ in Remark \ref{special sit} with  $v = i$, $a_1=j$, $a_2=k$ and then by Corollary \ref{coro for kronecker pairs1} we obtain a Kronecker pair.

We reduce to the case
\subsubsection{If there are not parallel arrows in \texorpdfstring{$Arr(Q)$}{\space}} In this case for any pair $i,j \in V(Q)$, $i \neq j$ we have  $\#(Arr(i,j))= \#(Ed(i,j)) \leq 1$. In this subsubsection the term  \textit{loop with $m$ vertices, $m\geq 1$  in $\Gamma(Q)$} means a sequence $a_1,a_2,\dots,a_m$ in $V(Q)$, s. t. $\#\{a_1,a_2,\dots,a_m\}=m$ and
$\{ Ed(a_i,a_{i+1})\neq \emptyset \}_{i=1}^{m-1}$,
$Ed(a_m,a_{1})\neq \emptyset$. Since there are not edges-loops and
there are not parallel arrows in $Q$, then any loop in $\Gamma(Q)$
(if there is such) must be with $m\geq 3$ vertices.

First we show quickly how to get a Kronecker pair if there are not loops.

If there are not loops in $\Gamma(Q)$ (recall also that by assumption $Q$ is neither Dynkin nor Euclidean) then one can  show that for some proper subset
$A\subset V(Q)$   $Q_A$  is an Euclidean quiver of the type $\wt{D}_m$, $m\geq 4$, $\wt{E}_6$, $\wt{E}_7$, $\wt{E}_8$. Take an adjacent to $Q_A$ vertex
 $v\in V(Q)$. The assumption that there are not loops in $\Gamma(Q)$ imply that there is a unique   edge between $v$ and $Q_A$, i. e. either $Arr(\{v\}, A)=\emptyset$ or  $Arr( A,\{v\})=\emptyset$, and then we can apply Corollary \ref{coro for kronecker pairs} to obtain a Kronecker pair.

Till the end of this subsubsection we assume that there   is a loop
in $\Gamma(Q)$.  Let us fix a loop with minimal number of vertices
$a_1,a_2,\dots,a_m$, i. e.  $m$ is the minimal possible number of
vertices in a loop. Denote $A=\{a_1,a_2,\dots,a_m\}$, $\#(A)=m$.
From the minimality of $m$ it follows that
$Ed(a_i,a_j)=\emptyset$, if $1\leq i<j\leq m$, $2\leq j-i \leq
m-2$, hence $Q_A$ (recall Definition \ref{subquiver}) is a quiver
with $\Gamma(Q_A)=\wt{A}_{m-1}$.  As above,
there exists an adjacent to $Q_A$ vertex $v \in V(Q)$. From Corollary
\ref{coro for kronecker pairs} it follows that we can assume
$Arr(\{v\}, A) \neq \emptyset$,  $Arr( A, \{v\}) \neq \emptyset$.
In particular $\#(Ed(A,\{v\}))\geq 2$. Let us summarize
\begin{gather} \Gamma(Q_A)=\wt{A}_{m-1}, \ \ \{v\} \cap A = \emptyset, \ \  Arr(\{v\}, A) \neq \emptyset, \ Arr( A, \{v\}) \neq \emptyset \nonumber\\[-2mm]
  \label{loop and vertex} \\[-2mm] m\geq \#(Ed(A,\{v\}))\geq 2, m\geq 3. \nonumber  \end{gather}
 We consider several cases depending on the numbers $m$, $\#(Ed(A,\{v\}))$.

{\it The case $\#(Ed(A,\{v\}))=m=3$.} \mbox{} \\ We can order $A=\{a_1,a_2,a_3\}$ so that $Q_A= \begin{diagram}[size=1em]
a_2  & \\
     &\rdTo \\
     &    &a_3 .\\
\uTo &\ruTo \\
a_1
\end{diagram}$   \  \  \  By $\#(Ed(A,\{v\}))=3$ it follows that  $v\in V(Q)$ must be connected to all the three vertices  $\{a_1,a_2,a_3\}$ and by
 \eqref{loop and vertex} one of the arrows must start at $v$ and another must end at it. We have either $Arr(\{v\}, \{a_2\}) \neq \emptyset$ or
 $Arr( \{a_2\},\{v\}) \neq \emptyset$.

If $Arr(\{v\}, \{a_2\}) \neq \phi$, then by the assumption that there are not oriented cycles we have
$Arr(\{a_3\},\{v\} ) = \phi$, $Arr(\{v\}, \{a_3\}) \neq \phi$, so that we can only choose  the direction of $Ed(v,a_1)$, i. e. we have two options
for $Q_{\{a_1,a_2,a_3,v\}}$:
$$  \begin{diagram}[1em]
a_2  &  & & & \\
     &\rdTo  \luTo(4,2)&  &  &\\
     &       & a_3 & \lTo & v\\
\uTo & \ruTo &     &    \ruTo(4,2)  & \\
a_1 & & & &
\end{diagram}  \qquad   \begin{diagram}[1em]
a_2  &  & & & \\
     &\rdTo  \luTo(4,2)&  &  &\\
     &       & a_3 & \lTo & v\\
\uTo & \ruTo &     &    \ldTo(4,2)  & \\
a_1 & & & &
\end{diagram} \ .$$
In both  the cases $a_3$ is a sink in $Q_{\{a_1,a_2,a_3,v\}}$, hence we can apply Corollary \ref{coro for kronecker pairs} to $Q_{\{a_1,a_2,v\}}$
 and $a_3\in V(Q)$ and get Kronecker pairs.

If $Arr( \{a_2\},\{v\}) \neq \phi$, now the direction of $Ed(v,a_1)$ is fixed and  both the  options for $Q_{\{a_1,a_2,a_3,v\}}$ are:
$$  \begin{diagram}[1em]
a_2  &  & & & \\
     &\rdTo  \rdTo(4,2)&  &  &\\
     &       & a_3 & \lTo & v\\
\uTo & \ruTo &     &    \ruTo(4,2)  & \\
a_1 & & & &
\end{diagram}  \qquad   \begin{diagram}[1em]
a_2  &  & & & \\
     &\rdTo  \rdTo(4,2)&  &  &\\
     &       & a_3 & \rTo & v\\
\uTo & \ruTo &     &    \ruTo(4,2)  & \\
a_1 & & & &
\end{diagram} \ .$$
In  the first  case we apply  Corollary \ref{coro for kronecker pairs} to $Q_{\{a_1,a_2,v\}}$ and $a_3\in V(Q)$  and in the second we apply it to
 $Q_{\{a_1,a_2,a_3\}}$ and $v\in V(Q)$.

{\it The case $\#(Ed(A,\{v\}))=2$, $m=3$.} \mbox{} \\
 Let us consider $\Gamma(Q_A)= \begin{diagram}[size=1em]
a_2  & \\
     &\rdLine \\
     &    &a_3 \\
\uLine &\ruLine \\
a_1
\end{diagram}$ \ \ without fixing the orientation.  Now there are only two edges between $A$ and $v$. Hence by \eqref{loop and vertex}  there are  exactly two
arrows between $A$ and $v$ and  one of them  must start at $v$ and the other must end at it.  As long as we have not fixed the orientations of the arrows in $Q_A$,
 we can assume that $Arr(\{a_1\},\{v\})\neq \emptyset$, $Arr(\{v\},\{a_2\})\neq \emptyset$. So  that $Q_{\{a_1,a_2,a_3,v\}}$, up to a choice of orientation in
 $Q_A$,  is $  \begin{diagram}[1em]
a_2  &  & & & \\
     &\rdLine  \luTo(4,2)&  &  &\\
     &       & a_3 & & v.\\
\uLine & \ruLine &     &    \ruTo(4,2)  & \\
a_1 & & & &
\end{diagram} $ \ We consider now the possible  choices of directions of the arrows in $Q_A$.  If     $a_3$ is a source/sink in $Q_A$, then it is a
source/sink in $Q_{A\cup\{v\}}$ and we can apply Corollary \ref{coro for kronecker pairs} to $Q_{\{a_1,a_2,v\}}$ and $a_3\in V(Q)$. Whence, by the condition
that  $Q$ is acyclic, we reduce to
$$  \begin{diagram}[1em]
a_2  &  & & & \\
     &\luTo  \luTo(4,2)&  &  &\\
     &       & a_3 & & v\\
\uTo & \ruTo &     &    \ruTo(4,2)  & \\
a_1 & & & &
\end{diagram}=\begin{diagram}[1.5em]
a_3 &\rTo &a_2 &\lTo &v\\
&\luTo &\uTo &\ruTo \\
& &a_1
\end{diagram} $$
and this is a permutation of the special case $S_2$ of  Remark \ref{special cases}. In this case we obtain a Kronecker pair by Corollary
\ref{coro for kronecker pairs1} applied to $Q_{\{v,a_2,a_3\}}$ and $a_1\in V(Q)$.



{\it The case  $m=4$.} \mbox{} \\
In this case $ \Gamma(Q_A)= \begin{diagram}[1.5em]
 a_4 &  \rLine  &  a_3   \\
  \uLine &        & \uLine     \\
a_1   & \rLine  &    a_2
\end{diagram} $ and by the minimality of $m=4$ it follows $\#(Ed(A,\{v\}))=2$ (recall that we have reduced to \eqref{loop and vertex}).
In other words the adjacent vertex $v$ must be connected to two of the vertices of the  quadrilateral $ \Gamma(Q_A) $. Again by the minimality of $m=4$
 these two vertices must be diagonal and, as long as we have not fixed the orientations of the arrows, we can assume that
 $Arr(\{a_1\},\{v\})\neq \emptyset$, $Arr(\{v\},\{a_3\})\neq \emptyset$. So  that $Q_{\{a_1,a_2,a_3,a_4,v\}}$, up to a choice of orientation in $Q_A$,  is
$ \begin{diagram}[1em]
& &a_3 & \\
&\ldLine &\dLine &\luTo \\
a_4 &    &a_2 &    &v.\\
&\rdLine &\dLine &\ruTo \\
& &a_1
\end{diagram} $ \ It follows to assign directions of the arrows in $Q_A$.  If    $a_4$ or $a_2$ is a source/sink in $Q_A$
then we can apply Corollary \ref{coro for kronecker pairs} to $Q_{\{a_1,a_2,a_3,v\}}$ and $a_4\in V(Q)$ or to $Q_{\{a_1,a_4,a_3,v\}}$ and $a_2\in V(Q)$,
 respectively. Whence, we reduce to
$Q_{\{a_1,a_2,a_3,a_4,v\}}= \begin{diagram}[1em]
& &a_3 & \\
&\ruTo &\uTo  &\luTo \\
a_4 &    &a_2 &    &v,\\
&\luTo &\uTo &\ruTo \\
& &a_1
\end{diagram}  $  which is a permutation of  the quiver $S_3$ in Remark \ref{special cases}.

{\it The case $m\geq 5$.}

If $m= 2 k + 1$ is odd, $k\geq 2$, then we can depict $\Gamma(Q_A)$ , $v\in V(Q)$ and one edge in $Ed(\{v\},Q_A)$ as follows:
$$\bd[1em] a_{k+1} & \rLine  &   &         & a_{k+2} \\
     \vdots &          &       &         & \vdots \\
      a_3   &          &       &         & a_{2k} \\
      \dLine &         &       &          & \dLine \\
      a_2    &         &       &         & a_{2k+1} \\
             & \rdLine &       & \ruLine &         \\
             &         & a_1   &         &         \\
             &         &\dLine &         &          \\
             &         & v     &         &        \\
\ed $$
We have reduced to the case $\#(Ed(\{v\},Q_A)) \geq 2$ (see \eqref{loop and vertex}). If we add another edge between $v$ and $A$ then we obtain another loop
with number of vertices less or equal to $k+2$. By $k\geq 2$ we have $k+2<2 k +1$, which contradicts the minimality of $m=2 k+1$.

If $m= 2 k $ is even, $k\geq 3$, then we can depict $\Gamma(Q_A)$ , $v\in V(Q)$ and one edge in $Ed(\{v\},Q_A)$ as follows:
$$\bd[1em] &         & a_{k+1} &       &        \\
            & \ruLine  &         & \luLine &        \\
    a_{k}   &          &          &         & a_{k+2} \\
     \vdots &          &       &         & \vdots \\
      a_3   &          &       &         & a_{2k-1} \\
      \dLine &         &       &          & \dLine \\
      a_2    &         &       &         & a_{2k} \\
             & \rdLine &       & \ruLine &         \\
             &         & a_1   &         &         \\
             &         &\dLine &         &          \\
             &         & v     &         &        \\
\ed $$
Again,  another edge between $v$ and $A$ produces  another loop with number of vertices less or equal to $k+2$. By $k\geq 3$ we have $k+2<2 k $,
which contradicts the minimality of $m=2 k$.

Proposition \ref{prop for kronecer pairs in quivers} is completely proved and  it implies  Proposition \ref{non euc and non dyn}.

Having Proposition \ref{non euc and non dyn}, Remark \ref{oriented cycles},  Corollary \ref{coro for density1} and Lemma \ref{Dynkin} we obtain table \ref{table}.

\subsection{Further examples of Kronecker pairs} \label{further examples}
Here we give some more examples of Kronecker pairs. In all the cases Corollary
\ref{coro for kronecker pairs 3} can be applied.
 \subsubsection{Markov triples}  It is shown in \cite[Example 3.2]{BondalPol}
 that if $X$ is a smooth projective variety (we assume over $\CC$), such that
 $D^b(Coh(X))$ is  generated by a strong exceptional collection of three
 elements (for example $X={\mathbb P}^2$) then for any such collection
 $(E_0,E_1,E_2)$ the dimensions $a=\dim(\Hom(E_0,E_1))$,
 $b=\dim(\Hom(E_0,E_2))$, $c=\dim(\Hom(E_1,E_2))$ satisfy Markov's equation
 $a^2+b^2+c^2 =  a b c$.  If $(a,b,c)\neq (0,0,0)$ and $a,b,c \leq 3$  then
 $(a,b,c)$  satisfy Markov's equation iff $a=b=c=3$, i. e.  the ``minimal'' such
 triple is $(3,3,3)$. Hence for any strong collection $(E_0,E_1,E_2)$ on
 $D^b(Coh(X))$  for some $i<j$ the pair $(E_i,E_j[-1])$ is Kronecker. Corollary
 \ref{coro for kronecker pairs 3} can be applied since $D^b(Coh(X))$ is of finite type.

 \subsubsection{$\PP^1 \times \PP^1$} In \cite[p. 3]{perling} a full exceptional
 collection  consisting of sheaves on $\PP^1 \times \PP^1$ is described. The
 matrix given there contains the dimensions of $\Hom(E_i,E_j)$, where $E_i,$
 $E_j$ are pairs in the exceptional collection. The number 4 in this matrix
 corresponds to a Kronecker pair.

\subsubsection{$\PP^n$, $n\geq 2$ and their blow ups}
Another example, where Corollary \ref{coro for kronecker pairs 3} can be applied,
is the standard strong exceptional collection $(\mc O, \mc O(1),\dots, \mc
O(n))$ on $\mathbb P^n$, $n\geq 2$. For $n\geq 2$ $\{ \dim(\Hom(\mc O(i-1),\mc
O(i)))\geq 3 \}_{i=1}^n$ so that $\{ (\mc O(i-1),\mc O(i)[-1])  \}_{i=1}^n$ are all Kronecker pairs.

Take now $\PP^n$, $n\geq 2$ and blow it up in finite number of points and let
the obtained variety be $X$. By  \cite[Theorem 4.2]{BondalOrlov}  we know that $D^b(X)$ has a
semiorthogonal decomposition $\left \langle E_1,E_2,\dots,E_l, D^b(\PP^n) \right
\rangle$, where $E_1,E_2,\dots,E_l$ are exceptional objects. The Kronecker
pairs of  $D^b(\PP^n)$ are also Kronecker pairs in $D^b(X)$ and Corollary
\ref{coro for kronecker pairs 3} can be applied. In particular these arguments
hold for all del Pezzo surfaces.

After  blowing up in a more general subvarieties $Y\subset \PP^n$ we still get
Kronecker pairs but in this case one must check the extendability condition in
Theorem \ref{coro for kronecker pairs 0} (see remark \ref{more general} ).

\section{Open questions}
In the previous sections we have described the foundations  of some  new
directions in the theory of derived categories and their connection with
``classical'' geometry and dynamical systems. We believe this is a promising
field and conclude by formulating possible future directions of research.

\subsection{Algebraicity of entropy}

In all examples we have considered so far the graph of the entropy function
$t\mapsto h_t(F)$ is an algebraic curve defined over $\mathbb Q$ in coordinates
$(\exp(t),\exp(h_t(F)))$.
Also, in the $\mathbb Z/2$-graded case $\exp(h_0(F))$ is an algebraic number.
(The stretch factor of a pseudo-Anosov map is an algebraic integer of degree
bounded above by a quantity depending only on the genus, see~\cite{flp}.)

\begin{q}
Is algebraicity of entropy a general phenomenon, and if so, what
are natural sufficient conditions for it?
\end{q}

\subsection{Pseudo-Anosov autoequivalences}

Recall that pseudo-Anosov maps are characterized by the fact that
they preserve a quadratic differential up to rescaling in the
horizontal and vertical direction by factors $1/\lambda$ and
$\lambda$ respectively. Conjecturally, such a quadratic
differential defines a stability condition on the (wrapped) Fukaya
category of the complement of its zeros. This stability condition
would then be preserved by the autoequivalence up to rescaling of
the real and imaginary parts of the central charge. In the case of
the torus this follows from homological mirror symmetry and the
description of the space of stability conditions of the bounded
derived category of an elliptic curve. With this in mind, we
propose the following.

\begin{df}
An autoequivalence $\phi$ of a triangulated category $\mc C$ is
\emph{pseudo-Anosov} if there exists a Bridgeland stability condition $\sigma$
on $\mc C$ and $\lambda>1$ such that
\begin{equation} \label{catPseudoAnosov}
\phi\sigma = \left( \begin{array}{cc} \lambda^{-1}
& 0 \\ 0 & \lambda \end{array}\right)\sigma
\end{equation}
where we use the action of the universal cover of $GL^+(2,\mathbb R)$ on
$\mathrm{Stab}(\mc C)$.
\end{df}

An example is given by the shifted Serre functor $S\circ[-1]$ on
the bounded derived category of the Kronecker quiver with $m\geq 3$ arrows.
A $\sigma$ satisfying \eqref{catPseudoAnosov} lies in the chamber where the set
of phases is dense in an arc.
The stretch factor is given by
\begin{equation}
\lambda = \frac{m^2+\sqrt{m^4-4m^2}}{2}-1
\end{equation}
and $\log\lambda$ is the entropy of $S\circ[-1]$.

Suppose $\mc C$ is the Fukaya-category of a symplectic manifold $M$ and $\phi$
is a pseudo-Anosov autoequivalence in the above sense.

\begin{q}
Are there, in analogy with the case of surfaces, transverse lagrangian
foliations of $M$ which are preserved by some symplectomorphism which induces
$\phi$?
\end{q}

\subsection{Birational maps and dynamical spectrum}

Recently   a new approach to birational geometry was suggested in the work of
Cantat and Lamy \cite{CL}. In their celebrated work they make a parallel between
dynamical theory  of Teichm\"uller space and the Bogomolov-Picard-Manin space
associated with an algebraic surface.
This connection was taken a step further in the work of Blanc and Cantat
\cite{BlancCantat} where the dynamical spectra of the groups of birational
automorphisms of surfaces acting on the Bogomolov-Picard-Manin space were
investigated.

In the previous sections we suggested that the space of stablity conditions
should play the role of such a  categorical Teichm\"uller space.

\begin{q}
Is there a categorical approach to Bogomolov-Picard-Manin spaces?
\end{q}

First we restrict  ourselves  to the case  $\dim(X)=2$.
Let $X$ be a smooth projective surface  over $\mathbb C$.
A birational map $\phi:X\to X$ does not induce a functor on $D^b(X)$ itself,
but (assuming it is an isomorphism in codimension one) on the quotient category
$D^b_{(1)}(X)=D^b(X)/S$, where $S$ is the subcategory of complexes with
cohomology supported in codimension at least two (see \cite{MeinhardtPartsch}).
On the other hand, $\phi$ has a well-defined dynamical degree $\lambda(\phi)$.

\begin{q}
Are these notions of entropy the same, i.e. is it true that
$\log\lambda(\phi)=h_0(\phi^*)$?
\end{q}

In analogy with the dynamical spectrum studied by Blanc and Cantat, one can,
quite generally, consider the set entropies of autoequivalences of a
triangulated category.
This is an invariant which deserves, in our opinion, further study.
A first natural question is:

\begin{q}
What is the relation between the dynamical spectrum of $X$ and the entropy
spectrum of $D^b_{(1)}(X)$?
\end{q}

We move now to the case when $dim(X)>2$. In this case the main question is:

\begin{q}
What is the right analogue of $D^b_{(1)}(X)$ and of the categorical
Bogomolov-Picard-Manin space in the case  $dim(X)>2$?
\end{q}

Once this question is answered we expect dynamical spectra of these categories
to play an important role in the study of birational geometry of $X$. In
particular, the rich geometry of Fano manifolds suggest that the dynamical
spectrum will contain significant gaps for many non-rational Fano manifolds.

\subsection{Complexity and mass}

Suppose $\mc T$ is a triangulated category with stability condition $\sigma$.
Let $E\in\mc T$ be an object with semistable factors $G_i$.
The \emph{mass} $m(E)$ of $E$ is by definition $\sum_i|Z(G_i)|$.
As with complexity, we can introduce a parameter $t$ and consider
\begin{equation}
m_t(E)=\sum_i|Z(G_i)|\re^{\frac{\phi(G_i)}{\pi}t}
\end{equation}
where $\phi(G_i)\in\mathbb R$ is the phase of $G_i$.
Intuitively, $m_t(E)$ measures the ``size'' of $E$, analogous to $\delta_t(G,E)$
for a fixed generator $G$.

\begin{q}
Is $m_t$ equivalent to $\delta_t(G,\_)$, in the same sense that the
$\delta_t(G,\_)$ are equivalent for various choices of $G$?
\end{q}

We hope to return to this question in future work.

\subsection{Questions related to Kronecker pairs and density of phases}

The results of Section 3 are a  motivation for the following   questions.

Recall that in any of  the quivers listed in Remark \eqref{special cases}  we found Kronecker pairs $(E_1,E_2)$ with
$\dim(\Hom^{k}(E_1,E_2))= 3$ for $k=1$ and $0$ for $k\neq 1$. We expect that the first part of  following question  has a positive answer:

\begin{q} Do the inequalities $\{ \dim(\Hom^{k}(E_1,E_2))\leq 3 \}_{k \in \ZZ}$ hold for any exceptional pair $(E_1,E_2)$  in any of  the
quivers listed in Remark \eqref{special cases} ?

 Determine all the quivers $Q$, s. t. the dimensions
$\{ \dim(\Hom^{k}(E_1,E_2)) \}_{k \in \ZZ} $, where $(E_1,E_2)$ vary through all
exceptional pairs,  are bounded above.
\end{q}

Recall that by Corollary \ref{upper bound for euc} these dimensions are
strictly smaller than $3$ if $Q$ is either  Euclidean or a Dynkin quiver, and
they  are not bounded above if  Corollary \ref{coro for kronecker pairs}  can
be applied to $Q$.

 We expect that  the categories with non-dense behaviour of phases form a ``thin'' set of categories.
Among these, the categories $D^b(Rep_\CC(Q))$ with $Q$ an Euclidean quiver  have somehow remarkable  behaviour of $P_\sigma$
(see the second row of Table \ref{table}).  So that the next question is:

\begin{q}  Which are the triangulated categories $\mc T$, s. t. for any $\sigma
\in \st(\mc T)$ the set of phases  $P_\sigma$ is either finite or has two limit points? \end{q}

    \end{document}